\pdfoutput=1
\RequirePackage{ifpdf}
\ifpdf % We are running pdfTeX in pdf mode
\documentclass[pdftex]{sigma}
\else
\documentclass{sigma}
\fi

\numberwithin{equation}{section}

\newtheorem{Theorem}{Theorem}[section]
\newtheorem{Corollary}[Theorem]{Corollary}
\newtheorem{Lemma}[Theorem]{Lemma}
\newtheorem{Proposition}[Theorem]{Proposition}
\newtheorem{Question}[Theorem]{Question}
{\theoremstyle{definition}

\newtheorem{Example}[Theorem]{Example}

\newtheorem{Remarks}[Theorem]{Remarks}
\newtheorem{Observations}[Theorem]{Obervations}
\newtheorem{Observation}[Theorem]{Observation}
\newtheorem{Examples}[Theorem]{Examples}
\newtheorem{Comments}[Theorem]{Comments}}

\begin{document}
\allowdisplaybreaks

\newcommand{\arXivNumber}{1812.02965}

\renewcommand{\thefootnote}{}

\renewcommand{\PaperNumber}{071}

\FirstPageHeading

\ShortArticleName{Stratified Bundles on Curves}

\ArticleName{Stratified Bundles on Curves and Differential Galois\\ Groups in Positive Characteristic\footnote{This paper is a~contribution to the Special Issue on Algebraic Methods in Dynamical Systems. The full collection is available at \href{https://www.emis.de/journals/SIGMA/AMDS2018.html}{https://www.emis.de/journals/SIGMA/AMDS2018.html}}}

\Author{Marius VAN DER PUT}

\AuthorNameForHeading{M.~van der Put}

\Address{Bernoulli Institute, University of Groningen, \\ P.O.~Box 407, 9700 AG Groningen, The Netherlands}
\Email{\href{mailto:m.van.der.put@rug.nl}{m.van.der.put@rug.nl}}

\ArticleDates{Received December 10, 2018, in final form September 14, 2019; Published online September 21, 2019}

\Abstract{Stratifications and iterative differential equations are analogues in positive characteristic of complex linear differential equations. There are few explicit examples of stratifications. The main goal of this paper is to construct stratifications on projective or affine curves in positive characteristic and to determine the possibilities for their differential Galois groups. For the related ``differential Abhyankar conjecture'' we present partial answers, supplementing the literature. The tools for the construction of {\em regular singular} stratifications and the study of their differential Galois groups are $p$-adic methods and rigid analytic methods using Mumford curves and Mumford groups. These constructions produce many stratifications and differential Galois groups. In particular, some information on the tame fundamental groups of affine curves is obtained.}


\Classification{14F10; 13N10; 14G22; 14H30}

\renewcommand{\thefootnote}{\arabic{footnote}}
\setcounter{footnote}{0}

\section{Introduction and summary}\label{section1}

Let $C$ denote an algebraically closed field of characteristic $p>0$. For an irreducible smooth algebraic variety $Y$ over $C$, we write $\mathcal{D}_{Y/C}$ for the sheaf of differential operators on $Y$ (see \cite[Section~16.8.1]{EGA4}). A {\it stratified bundle} (also called a stratification) is a locally free $O_Y$-module of finite rank equipped with a compatible left action by $\mathcal{D}_{Y/C}$. {\it Iterative differential modules} were introduced and studied in~\cite{Ma-vdP1}. We briefly indicate the relation with stratified bundles. See Section~\ref{section2.1} for details.

Consider the case $\dim Y=1$. The field $C(Y)$ is provided with a {\it higher derivation} $\big\{\partial^{(n)}\big\}_{n\geq 0}$ such that $\partial^{(n)}\neq 0$ for all $n$. Then $C(Y)\big[\big\{\partial^{(n)}\big\}_{n\geq 0}\big]$ is the algebra of the $C$-linear differential operators on~$C(Y)$. An {\it iterative differential module}~$N$ is a finite-dimensional $C(Y)$-vector space equipped with a left action $\partial_N^{(n)}$ of $\partial^{(n)}$ for all $n\geq 0$ satisfying the rules corresponding to the statement that $N$ is a left $C(Y)\big[\big\{\partial^{(n)}\big\}_{n\geq 0}\big]$-module.

\looseness=-1 Let $M$ be a stratification on $Y$. The generic fiber $M_\eta$ is a finite-dimensional $C(Y)$-vector space with a left action of the algebra of $C$-linear differential operators of~$C(Y)$. Thus $M_\eta$ is an iterative differential module. Moreover one can reconstruct $M$ from $M_\eta$ by considering the regular points for $M_\eta$. In order to avoid pathologies we {\it require} that any iterative differential module $N$ over $C(Y)$ has the form $M_\eta$ for a stratification $M$ on some (non-empty) open subset of $Y$.

The study of stratified bundles started with D.~Gieseker's paper \cite{Gie} inspired by N.~Katz. This has led to work by H.~Esnault, V.~Metha, L.~Kindler, J.P.P.~dos~Santos, I.~Biswas et al., see for instance \cite{Sa, Es,Es-M,K,Ki}. Iterative differential modules are studied by B.H.~Matzat, A.~Maurischat, S.~Ernst et al., see for instance \cite{Er, Ma-vdP2,R}.

In the complex context, the {\it Riemann--Hilbert correspondence} is the relation between representations of the topological fundamental group and connections. For a smooth irreducible algebraic variety $Y$ over $C$ and a point $y_0\in Y(C)$, there is an \'etale fundamental group $\pi^{\rm et}_1(Y,y_0)$ and a tame fundamental group $\pi_1^{\rm tame}(Y,y_0)$. The latter is a canonical quotient of $\pi^{\rm et}_1(Y,y_0)$. Both groups are ``too small'' for describing the stratified bundles on $Y$. Let $\operatorname{Strat}(Y)$ denote the category of the stratified bundles on $Y$. The point $y_0\in Y(C)$ induces a fiber functor and this provides $\operatorname{Strat}(Y)$
with the structure of a neutral Tannakian category, equivalent to the category of the representations of an algebraic group scheme $\pi^{\rm str}(Y,y_0)$, called the {\it stratified fundamental group}. The group scheme $\pi^{\rm str}(Y,y_0)$ is an analogue of the algebraic hull $\pi_1^{\rm top}(Y)^{\rm hull}$ of the topological fundamental group in the complex case. The group of the connected components of $\pi^{\rm str}(Y,y_0)$ is $\pi^{\rm et}_1(Y,y_0)$.

For an object $M$ of $\operatorname{Strat}(Y)$, one considers the full Tannakian subcategory $\{\{M\}\}$ generated by~$M$. Its Tannakian group is a linear algebraic group over $C$. These groups and also $\pi^{\rm str}(Y,y_0)$ are {\it reduced}, according to~\cite{Sa}.

For the case $\dim Y=1$, we consider the iterative differential module $M_\eta$ and its Tannakian category $\{\{M_\eta\}\}$. The natural functor $\{\{M\}\}\rightarrow \{\{M_\eta \}\}$ is an equivalence of Tannakian categories (see Lemma~\ref{2.*}). From Picard--Vessiot theory one obtains that the Tannakian group for~$\{\{M_\eta \}\}$ coincides with the differential Galois group (see~\cite{Ma-vdP1} and Section~\ref{section2.1}). This dif\-fe\-ren\-tial Galois group is known to be a {\it reduced} linear algebraic group over $C$. This group will also be called the differential Galois group of~$M$.

In the sequel we will only consider linear algebraic groups and affine algebraic group schemes that are {\it reduced} and we drop the term reduced.

In contrast to the complex case, it is rather difficult to produce stratified bundles. The main goal of this paper is to develop various methods for the construction of stratifications on $X\setminus S$, where~$X$ is a curve (smooth, projective, irreducible) of genus~$g$ over~$C$ and $S\subset X$ a finite set.

In almost all cases it seems impossible to determine a stratified fundamental group. The determination in \cite[Section~3]{Sa} (see Theorem~\ref{2.4}) of the stratified fundamental group of an abelian variety is an exception. Instead of determining $\pi^{\rm str}(X\setminus S)$ we try to determine the linear algebraic groups which are {\it realizable for~$(X,S)$}, i.e., which are quotients of $\pi^{\rm str}(X\setminus S)$.

 One considers for a linear algebraic group $G$ the subgroup $p(G)$ generated by all elements of~$G$ with order $p^m$ for some $m\geq 0$. This is a Zariski closed normal subgroup of~$G$. Let $S\subset X$ be a finite, non empty set. We try in Section~\ref{section2}, as in~\cite{Ma-vdP1}, to find answers for:

\begin{Question}\label{Question1} Suppose that $H$ is a linear algebraic group with $p(H)=\{1\}$. When is $H$ rea\-li\-zable for $(X,S)$?
\end{Question}

\begin{Question}\label{Question2} Is the group $G$ realizable for $(X,S)$ if this holds for $G/p(G)$?
\end{Question}

We will call Question~\ref{Question2} {\it the differential Abhyankar conjecture}. A finite group $G$ is realizable for a stratification on $X\setminus S$ if and only if~$G$ is Galois group of an \'etale covering of $X\setminus S$. Thus the above two questions for finite groups coincide with Abhyankar's well known questions. According to~\cite{SGA1}, {\it the answer to Question~{\rm \ref{Question1}} for a finite group $H$ is}:
\begin{enumerate}\itemsep=0pt
\item[(a)] {\it Let $g$ be the genus of $X$ and let $s\geq 0$ denote the minimal number of generators of $H$. Then $H$ is realizable if and only if $s\leq 2g-1+\#S$}.
\end{enumerate}
{\it For infinite groups $H$}, the condition $p(H)=\{1\}$ is equivalent to $H^o$ being a torus and the order of $H/H^o$ being prime to~$p$. Let $s\geq 0$ be the minimal number of generators of $H/H^o$. Theorem~\ref{2.1} improves the partial answers in \cite[Theorems~7.1 and~7.2]{Ma-vdP1} as follows:
\begin{enumerate}\itemsep=0pt
\item[(b)] {\it An infinite commutative $H$ is realizable if and only if $s\leq 2g-1+\#S$ except in the two cases
 $g=0$, $\#S=1$ and $g>0$, $\# S=1$, $C=\overline{\mathbb{F}_p}$ and $T_p(\operatorname{Jac}(X))=0$ where $\operatorname{Jac}(X)$ denotes the Jacobian variety of~$X$ and $T_p(\operatorname{Jac}(X))$ its $p$-adic Tate module.}
\item[(c)] {\it For an infinite non commutative group $H$ the following conditions are sufficient for reali\-za\-bility on $(X,S)$}: {\it $s\leq 2g-1+\#S$ and
 the rank of the abelian group $O(X\setminus S)^*/C^*$ is greater than or equal to $\dim H^o$}.
\end{enumerate}

{\it Question~{\rm \ref{Question2}} for finite groups $G$} was given a positive answer by M.~Raynaud (for the case $\mathbb{P}^1\setminus \{\infty \}$) and by D.~Harbater for the general case~$X\setminus S$.

{\it Question~{\rm \ref{Question2}} for connected groups $G$} has a positive answer in \cite[Corollary~7.7]{Ma-vdP1}. The proof of part~(3) of Corollary~7.7 is however rather sketchy. We note that a positive answer to Question~\ref{Question2} for {\it connected} $G$ follows at once from \cite[Theorem 9.12] {R} where a detailed proof is given of the statement:

{\it Let $G$ be a connected linear algebraic group. Then $G$ can be realized for $(X,S)$ if $\#S\geq 2$. If the Jacobian variety of~$X$ has a non-trivial $p$-torsion point or if~$G$ is unipotently generated, then $G$ can be realized if $\#S=1$.}

The literature on {\it Question~{\rm \ref{Question2}} for non connected infinite groups} $G$ is not so clear and contains mistakes. We present here and in Section~\ref{section2} what can be proved at present. The condition $G/p(G)$ is realizable over $(X,S)$ implies that $G/G^o$ is realizable over $(X,S)$ and produces a~Galois covering $\big(\tilde{X},\tilde{S}\big)\rightarrow (X,S)$ with group $G/G^o$. Above $\big(\tilde{X},\tilde{S}\big)$ one has to realize $G^o$ in a~special (equivariant) way. This is an embedding problem.

\cite[Proposition 8.7]{Ma-vdP2} states that the embedding problem has a proper solution for $X=\mathbb{P}^1$ and any~$S$. Furthermore \cite[Theorem~8.11]{Ma-vdP2} claims that Question~\ref{Question2} has a positive answer for~$\big(\mathbb{P}^1,\{\infty \}\big)$. Both assertions are wrong since they are in contradiction with a {\it negative answer to Question~{\rm \ref{Question2}}} for the infinite dihedral group $\mathbb{D}_\infty$, $(X,S)=\big(\mathbb{P}^1, \{\infty \}\big)$ and $C=\overline{\mathbb{F}_2}$ found by A.~Maurischat \cite[Theorem~9.1]{R}.

In Section~\ref{section2.3} (Proposition~\ref{2.6} and Corollary~\ref{2.7}) we explain this example and extend it to any characteristic for $\big(\mathbb{P}^1,\{\infty \}\big)$
with $C=\overline{\mathbb{F}_p}$. These ``counterexamples'' disappear if one either replaces $\{\infty \}$ by $\{0,\infty \}$ or $\overline{\mathbb{F}_p}$ by a larger field. We would like to point out that the weaker versions of \cite[Proposition~8.7, Theorems~8.8 and~8.11]{Ma-vdP2} where one allows to increase the finite set~$S$, are valid.

Essentially the only tool for handling Question~\ref{Question2} is the construction of ``projective systems'' (flat bundles, $F$-divided sheaves) on (a covering of) $X\setminus S$ which produces a stratification with the required differential Galois group, see \cite[Theorem~1.3]{Gie}, \cite[Proposition~5.1]{Ma-vdP1} and \cite[Theorem~8]{Sa}. Finally, we note that there is no complete answer for the special case of Question~\ref{Question2}:
\begin{enumerate}\itemsep=0pt
\item[]{\it Which non-connected $G$ with $p(G)=G$ are realizable for $\big(\mathbb{P}^1,\{\infty \}\big)$?}
\end{enumerate}

In Section~\ref{section3} regular singular stratifications are investigated. {\it Regular singular stratifications}~$M$ on $X\setminus S$ are defined as follows. If $S=\varnothing$, then $M$ is considered to be regular singular. For $S\neq \varnothing$ one requires that~$M$ extends to a vector bundle $M^+$ on $X$ and that for every point $s\in S$ there is an affine neighbourhood $U$ of $s$ and a local parameter $t$ at $s$ with the properties $t^n\partial _t^{(n)}M^+(U)\subset M^+(U)$ for all $n$.

This makes sense, since the restriction $M^+(U)\rightarrow M^+(U\setminus S)=M(U\setminus S)$ is injective and becomes an isomorphism after tensoring with $C(X)$ over~$O(U)$.
 The above definition is a 1-dimensional version of the definition given in~\cite{Gie} for any dimension.

The regular singular stratifications on an affine curve $X\setminus S$ form also a neutral Tannakian ca\-tegory $\operatorname{Strat}^{\rm rs}(X\setminus S)$ and produce the regular singular stratified fundamental group \mbox{$\pi^{\rm str,rs}(X\setminus S)$} which is a quotient of $\pi^{\rm str}(X\setminus S)$. The group of the components
$\pi^{\rm str,rs}(X\setminus S)/\pi^{\rm str,rs}(X\setminus S)^o$ equals the tame fundamental group $\pi_1^{\rm tame}(X\setminus S,x_0)$. See~\cite{K} and~\cite{Ki} for details.

The local theory of regular singularities and local exponents is briefly exposed in Section~\ref{section3}. In \cite{Gie} the regular singular stratifications on $\mathbb{P}^1\setminus \{0,\infty \}$ are described and this leads in Section~\ref{section3.1} to the result $\pi^{\rm str,rs}\big(\mathbb{P}^1\setminus \{0,\infty\}\big)=\operatorname{Diag}(\mathbb{Z}_p/\mathbb{Z})$. This is a very small group compared to $\pi^{\rm str}\big(\mathbb{P}^1\setminus \{0,\infty\}\big)$. Furthermore the method of ``constructing'' a stratification on $X\setminus S, S\neq \varnothing$ by producing a projective system does not seem to work for regular singular stratifications.
In the special case $\mathbb{P}^1\setminus \{0,1,\infty\}$, the only method that we know to obtain regular singular stratifications is reducing ``bounded $p$-adic differential equations'' modulo~$p$.

\looseness=-1 Theorem~\ref{3.6} in Section~\ref{section3.2} improves \cite[Theorem~8.9]{Ma-vdP1}. It produces a family of regular singular stratifications of dimension two with local exponents (up to integers) $0$, $1-\gamma || 0$, $\gamma -\alpha -\beta || \alpha$, $\beta$ at the points $0$, $1$, $\infty$ for certain triples $(\alpha , \beta , \gamma) \in \mathbb{Z}_p^3$. These can be seen as analogues of the classical hypergeometric differential equations~$_2F_1$. The corresponding differential Galois groups are (roughly speaking) the reduction modulo~$p$ of the differential Galois groups in characteristic zero.

In Section~\ref{section3.3}, Proposition~\ref{3.7} it is shown that if the group $G$ is a differential Galois group for a regular singular stratification on $X\setminus S$, then $G^o$ is generated by its maximal tori and (of course) $G/G^o$ is a quotient of the tame fundamental group. In particular one can show, by using~\cite[Corollary~16]{Sa}, that the group $\mathbb{G}_a$ is not a differential Galois group for a regular singular stratification for any $X$ and any $S$.

In Section~\ref{section4} the uniformization $\Omega \rightarrow X$ of a Mumford curve over a field $C$, complete with respect to a non trivial valuation, is used to construct stratifications on~$X$. This is a rigid analytic analogue of the complex Riemann--Hilbert correspondence. This method is also present in \cite{Sa}. One concludes that
a linear algebraic group $G$ is a differential Galois group for $X$ if $G$ is topologically generated (for the Zariski topology) by $\leq g$ elements (see Theorem~\ref{4.2}).

In Section~\ref{section5} the field $C$ is again complete with respect to a non trivial valuation. Mumford groups $\Gamma\subset {\rm PGL}_2(C)$ are introduced. These are analogues of complex triangle groups and Kleinian groups. A representation $\rho\colon \Gamma \rightarrow {\rm GL}_n(C)$ produces a stratification on $\mathbb{P}^1_C$ which has
singularities at the set of branch points $S\subset \mathbb{P}^1_C$ of $\Gamma$. The singularities are regular singular if and only if $\Gamma$ is tame, i.e., $\Gamma$ has no elements of order $p$. For a tame Mumford group the number of branch points is at least~4. In particular, these groups cannot be used to construct ``hypergeometric stratifications''. There are a few tame Mumford groups with branch points $S=\{0,1,\infty ,\lambda \}$ and $0<|\lambda |<1$. In Corollary~\ref{5.5} these groups are used to obtain the following ``approximate answer'' to the inverse problem on $\mathbb{P}_C^1\setminus S$:

{\it Let $G$ be a linear algebraic group over $C$ such that $G^o$ is generated by its maximal tori. There are linear algebraic groups $G_3\lhd G_2\lhd G_1$ with $[G_ 1:G_2]< \infty$, such that $G_1$ is the differential Galois group of a stratification on $\mathbb{P}^1_C$ with regular singularities in~$S$ and $G\cong G_2/G_3$.

{\em For the special case of finite groups, one has}: Any finite group $G$ is a subquotient of the tame fundamental group of $\mathbb{P}^1\setminus \{0,1,\lambda ,\infty \}$ over $\overline{\mathbb{F}_p(\lambda )}$ (with transcendental~$\lambda$).}

\begin{Proposition}[the inverse problem depends on the base field]\label{1.1}
Let $C_0\subset C$ be algebraically closed fields of characteristic $p>0$. Consider the curves
$Y_0=X_0\setminus S$ over $C_0$ and $Y=C\times_{C_0}Y_0$. The functor $\mathcal{F}\colon \operatorname{Strat}^{\rm rs}(Y_0)\rightarrow \operatorname{Strat}^{\rm rs}(Y)$ is defined by $A\mapsto C\otimes_{C_0}A$ for the objects and morphisms $A$ of $\operatorname{Strat}^{\rm rs}(Y_0)$.
\begin{enumerate}\itemsep=0pt
\item[$(1)$] $C\otimes _{C_0}\operatorname{Hom}(A_0,B_0)\rightarrow \operatorname{Hom}(\mathcal{F}A_0,\mathcal{F}B_0)$ is an isomorphism for all $A_0,B_0\in \operatorname{Strat}^{\rm rs}(Y_0)$.
\item[$(2)$] If $G$ is the differential Galois group of an object $A_0\in \operatorname{Strat}^{\rm rs}(Y_0)$, then $C\times_{C_0}G$ is the differential Galois group of $\mathcal{F}A_0$.
\item[$(3)$] In general, an object $A\in \operatorname{Strat}^{\rm rs}(Y)$ does not descend to $C_0$, i.e., it does not lie in the image of $\mathcal{F}$. However $A$ descends to an extension $\tilde{C}$ of countable transcendence degree over $C_0$.
\end{enumerate}
\end{Proposition}

We postpone the proof to the end of Section~\ref{section5}. We remark that a more general result is proven in \cite{Ba}. Proposition~\ref{1.1} can be seen as a version in positive characteristic of a result \cite[Proposition~1.3.2]{Ka} by O.~Gabber. We observe in connection with Proposition~\ref{1.1} that for ``large fields''~$C$, say, of infinite
transcendence degree over~$\mathbb{F}_p$, the methods of this paper produce many stratifications. However, for the base field $C_0=\overline{\mathbb{F}}_p$ it seems that
$\operatorname{Strat}^{\rm rs}(X_0\setminus S)$ has few objects.

\section{Differential Galois groups for curves}\label{section2}
\subsection{Iterative differential modules and stratifications}\label{section2.1}

We recall the definition of iterative differential module and clarify the relation with stratified bundles. Let~$C$ denote an algebraically closed field of characteristic $p>0$. Let $Y$ be a smooth irreducible curve over~$C$ (affine or not). On the function field $C(Y)$ one considers a~{\it higher derivation} $\big\{\partial ^{(n)}\big\}_{n\geq 0}$ (over~$C$). This is a set of $C$-linear maps $\partial^{(n)}\colon C(Y)\rightarrow C(Y)$ satisfying:
\begin{enumerate}\itemsep=0pt
\item[(i)] $\partial^{(0)}$ is the identity;
\item[(ii)] $\partial ^{(n)}(fg)=\sum\limits_{a+b=n}\partial^{(a)}(f)\partial^{(b)}(g)$;
\item[(iii)] $\partial ^{(n)}\circ \partial ^{(m)}={n+m\choose n}\partial ^{(n+m)}$.
\end{enumerate}
We call the higher derivation {\it good} if moreover (iv) $\partial^{(n)}\neq 0$ for all $n$. In the sequel we only consider good higher derivations.

It is not difficult to show that the left $C(Y)$-vector space of all $C$-linear differential operators on the field~$C(Y)$ has basis $\big\{\partial ^{(n)}\big\}_{n\geq 0}$.

\begin{Example} Suppose that $C(t)\subset C(Y)$ is a (finite) separable extension. A standard higher derivation $\big\{\partial^{(n)}_t\big\}$ on~$C(t)$ is defined by the formulas $\partial_t^{(m)}(t^n)={n\choose m }t^{n-m}$. This standard higher derivation extends uniquely to~$C(Y)$ and will again be denoted by $\big\{\partial^{(n)}_t\big\}$. It is a good higher derivation.
\end{Example}

An {\it iterative differential module} $M$ is, in the terminology~\cite{Ma-vdP1}, a~finite-dimensional vector space over~$C(Y)$, equipped with $C$-linear operators $\big\{\partial_M^{(n)}\colon M\rightarrow M\big\}_{n\geq 0}$ satisfying:
\begin{enumerate}\itemsep=0pt
\item[(a)] $\partial_M^{0)}$ is the identity;
\item[(b)] $\partial^{(n)}_M(fm)=\sum\limits _{a+b=n}\partial^{(a)}(f)\partial_M^{(b)}(m)$ and
\item[(c)] $\partial _M^{(n)}\circ \partial _M^{(m)}={n+m\choose n }\partial_M^{(n+m)}$.
\end{enumerate}

Equivalently, $M$ is a left module over the algebra of all $C$-linear differential operators on~$C(Y)$. In order to {\it avoid pathological examples} (see \cite[Sections~4.2 and~7]{Ma-vdP1}) we {\it require} that there is a~$C(Y)$-basis $b_1,\dots ,b_d$ of $M$ and a non-empty affine open $U\subset Y$ such that $O(U)$ is stable under all~$\partial^{(n)}$ and $O(U)b_1+\cdots +O(U)b_d$ is stable under all $\partial_M^{(n)}$. In the sequel we just write $\partial^{(n)}$ for~$\partial^{(n)}_M$.

Let $M$ be a stratified bundle on $Y$. The generic fibre $M_\eta$ of~$M$ is a finite-dimensional vector space over~$C(Y)$. The left action of $\mathcal{D}_{Y/C}$ on~$M$ induces an action of the algebra of differential operators of~$C(Y)$ over~$C$ on~$M_\eta$. This algebra has basis $\big\{\partial ^{(n)}\big\}_{n\geq 0}$ over $C(Y)$ and so $M_\eta$ is an iterative differential module over~$C(Y)$.

Suppose now that an iterative differential module $N$ over~$C(Y)$ is given. Suppose that the curve~$Y$ is complete. There is, by definition, a stratification~$M$ on some non-empty affine open $U\subset Y$ with $M_\eta \cong N$ as iterative differential modules. There is a largest open subset $V\supset U$ such that~$M$ extends there as stratification, namely the set of the regular points of~$N$.

A point $y\in Y$ is called {\it regular} (for $N$) if there exists a basis $b_1,\dots ,b_d$ of $N$ over $C(Y)$ and an affine open neighbourhood $U$ of $y$ such that $O(U)$ and $O(U)b_1+\dots +O(U)b_d$ are stable by all $\partial ^{(n)}_t$, where $t$ is a local parameter at $y$ (see \cite[Section~7 and Corollary~6.2(3)]{Ma-vdP1}).

A point $y\in Y$ is {\it regular singular} (for $N$) if (with the same notations) $O(U)$ is stable under all $\partial^{(n)}_t$ and $O(U)b_1+\dots +O(U)b_d$ is stable under all $t^n\partial_t^{(n)}$. We note that the definition in Section~\ref{section1} of regular singular for a stratification $M$ coincides with this definition for its generic fiber $M_\eta$ as iterative differential module.

A stratified bundle $M$ of rank $d$ can be more explicitly described via the corresponding iterative differential module $M_\eta$. A choice of a basis for $M_\eta$ produces a~sequence of equations $\big\{\partial ^{(n)}y=A_ny\big\}_{n\geq 0}$, where~$y$ denotes a vector of length $d$ over $C(Y)$ and the $A_n$ are $(d\times d)$-matrices with entries in~$C(Y)$.

Let $M_\eta$ be written in matrix form $\big\{ \partial^{(n)}y=A_ny\big\}_{n\geq 0}$. A {\it Picard--Vessiot field} for $M_\eta$ is an extension of fields with higher derivations $L\supset C(Y)$ such that there exists a~fundamental matrix $F\in {\rm GL}_d(L)$ (this means that $F$ is a matrix of solutions $ \partial^{(n)}F=A_nF$ for all~$n$); the field of constants of $L$ is~$C$ and~$L$ is generated over $C(Y)$ by the entries of $F$ and $\frac{1}{\det F}$. A Picard--Vessiot field exists and is unique up to isomorphism.

The {\it differential Galois group} of $M_\eta$ is the group of the $C(Y)$-linear automorphisms of $L$ respecting the higher derivation. This group has a natural structure of {\it reduced linear algebraic group over $C$} (see \cite[p.~8]{Ma-vdP1}). Furthermore the differential Galois group coincides with the Tannaka group of the category $\{\{ M_\eta \}\}$ generated by~$M_\eta$.

\begin{Lemma}\label{2.*} Let $M$ be a stratification on $Y$ and $\{\{M\}\}$ the full Tannaka subcategory of $\operatorname{Strat}(Y)$ generated by $M$. Let $\{\{M_\eta \}\}$ denote the full Tannaka subcategory of the category of the iterative differential modules over $C(Y)$, generated by the object $M_\eta$. The natural functor $\{\{M\}\}\rightarrow \{\{M_\eta \}\}$ is an equivalence of Tannaka categories.
\end{Lemma}
\begin{proof} It is easily seen that we may restrict to the case that $Y$ is affine and the objects in the two categories are projective modules and vector spaces.

 (a) For $A,B\in \{\{M\}\}$ we have to show that $\operatorname{Hom}(A,B)\rightarrow \operatorname{Hom}(A_\eta ,B_\eta)$ is a bijection. The map is clearly injective. It suffices to consider the case $A=\bf{1}$ (the 1-dimensional trivial object). Indeed, $\operatorname{Hom}(A,B)\cong \operatorname{Hom}({\bf 1}, A^*\otimes B)$. Let $\mathcal{D}^+_{Y/C}\subset \mathcal{D}_{Y/C}$ denote the ideal of the sections $\mathcal{L}$ with $\mathcal{L}(1)=0$. We may write $(\mathcal{D}_{Y/C})_\eta=\oplus _{n\geq 0}C(Y) \partial ^{(n)}$ and $\big(\mathcal{D}^+_{Y/C}\big)_\eta=\oplus _{n\geq 1}C(Y) \partial ^{(n)}$ for a~suitable higher derivation.

Now $\operatorname{Hom}({\bf 1},B)=\big\{ b\in B\,|\, \partial ^{(n)}b=0 \ \text{for} \ n\geq 1\big\}$ and $\operatorname{Hom}({\bf 1}_\eta,B_\eta)=\big\{\xi \in B_\eta \, | \, \partial ^{(n)}\xi =0 \allowbreak \text{for} \ n\geq 1\big\}$. Let $\xi$ belong to the last group. Then $\{s\in O(Y)\, | \, s\xi \in B\}$ is a non-trivial differential ideal in~$O(Y)$ and therefore equal to~$O(Y)$. Thus $\xi$ lies in the first group and the map is surjective.

(b) In order to show that the functor is essentially surjective on objects, it suffices to show that for any object $A\in \{\{M\}\}$, any object $N\subset A_\eta$ is isomorphic to $B_\eta$ for some object $B\subset A$. Since $A\rightarrow A_\eta$ is injective we may consider $A$ as subset of $A_\eta$. Then $B:=A\cap N$ has the required property. \end{proof}

 We note that Lemma~\ref{2.*} also follows from \cite[Lemma~2.5(a)]{Ki}.

\begin{Proposition}\label{2.2} If the differential Galois group $G$ of a~stratified bundle on a curve $X$ satisfies $p(G)=\{1\}$, then every singular point is regular singular.
\end{Proposition}

\begin{proof} The global differential Galois group $G$ contains the local differential Galois group at a~singular point $s\in X$. Indeed, the Picard--Vessiot field for the stratified bundle over~$C(X)$ can be embedded into the Picard--Vessiot field of the same bundle but now over $C(X)_s$, the field of fractions of~$\widehat{O}_{X,s}$. According to~\cite{Ma-vdP1}, the local differential Galois group has a non trivial unipotent part if the singularity is irregular (i.e., not regular singular). Finally, unipotent elements have order a power of~$p$.
\end{proof}

\subsection{Question~\ref{Question1} for non-commutative infinite groups}\label{section2.2}

 For a non-commutative linear algebraic group $H$ such that $H^o$ is a torus and the order of $H/H^o$ is prime to $p$,
 we want to find {\it sufficient conditions} for $H$ to be realizable for $(X,S)$ where $X$ has genus $g$ and $S\neq \varnothing$.
 Using \cite[Lemma 5.11]{B-S} one can restrict to the case that $H$ is a semi-direct product of a torus $\mathbb{T}$ and a finite group $F$ of order prime to $p$. We note that the group $O(X\setminus S)^*/C^*$ is torsion free and finitely generated.

 \begin{Theorem}\label{2.1} Let $H$ be a semi-direct product of a torus $\mathbb{T}$ of dimension $d$ and a finite group
 $F$ generated by $s$ elements and with order prime to $p$.
 $H$ is realizable for $(X,S)$ if $s\leq 2g-1+\# S$ and
 $O(X\setminus S)^*/C^*$ has rank $\geq d$.
 \end{Theorem}
 \begin{proof} The proof is rather involved. We start by working out the details for an explicit example. After that we sketch the proof for the general case.

{\bf An example.} $(X,S)=\big(\mathbb{P}^1,\{0,1,\infty \}\big)$, $p\neq 3$, $H$ is the semi-direct product of the torus~$\mathbb{T}$ with character group $\mathbb{Z}^2$ and $C_3$. A~generator of $C_3$ is called $\rho$ and it acts on the character group~$\mathbb{Z}^2$ by the matrix ${0\ -1\choose 1\ -1 }$.

First we construct a field extension $K$ over $C(z)$ on which the group $H$ acts in a natural way. This field will be the Picard--Vessiot field of the, to be constructed, stratified bundle. We choose $K$ to be $C(u)(e_1,e_2)$, where $u^3=z$. Fix a $\omega \in C$ with $\omega ^3=1$, $\omega \neq 1$ The extension $C(u)\subset C(u)(e_1,e_2)$ is, by definition, purely transcendental. The group of automorphisms of~$K/C(z)$ that we consider is generated by $\rho$ with $\rho (u)=\omega u$, $\rho (e_1)=e_2$, $\rho(e_2)=e_1^{-1}e_2^{-1}$ and, for $t=(t_1,t_2)\in \mathbb{T}=(C^*)^2$, by the automorphism $\sigma_t(u)=u$, $\sigma_t(e_1)=t_1e_1$, $ \sigma_t(e_2)=t_2e_2$. One sees that this group is indeed the above~$H$.

Now we want to define on $K$ suitable higher derivations, extending the canonical higher derivations on $C(z)$. As explained in \cite[Section~2.1]{Ma-vdP1}, this we do by producing a suitable $C$-linear homomorphism $\phi=\phi _T\colon K\rightarrow K[[T]]$ which has the properties:
\begin{enumerate}\itemsep=0pt
\item[(a)] $\phi_T (a)\equiv a \ \text{mod}\ (T)$ for all $a\in K$,
\item[(b)] $\phi _{S+T}=\phi_T \circ \phi_S$,
\item[(c)] $\phi_T(z)=z+T$,
\item[(d)] $\phi_T$ is equivariant for the action of $H$.
\end{enumerate}
Property (c) says that the higher derivations on $K$ extend the canonical higher derivations on~$C(z)$. Because $u^3=z$ one has $\phi(u)=u\big(1+z^{-1}T\big)^{1/3}$. We think of $e_1$ as the symbolic expres\-sion $(1-u)^{a_0}(1-\omega u)^{a_1}\big(1-\omega ^2u\big)^{a_2}$ with (still to be chosen) $a_0,a_1,a_2\in \mathbb{Z}_p$ with $a_0+a_1+a_2=0$. Then $e_2=\rho e_1$ should ``symbolically'' be $(1-\omega u)^{a_0}\big(1-\omega^2 u\big)^{a_1}(1-u)^{a_2}$.

We complete the definition of $\phi=\phi_T$ by sending $e_1$ to the expression
\begin{gather*}
(1-\phi(u))^{a_0}(1-\omega \phi(u))^{a_1}\big(1-\omega^2\phi(u)\big)^{a_2}.
\end{gather*} This means that we define $\phi (e_1)$ by
 \begin{gather*} \phi (e_1)=e_1\cdot \frac{(1-\phi(u))^{a_0}(1-\omega \phi(u))^{a_1}\big(1-\omega^2\phi(u)\big)^{a_2}}{(1-u)^{a_0}(1-\omega
 u)^{a_1}\big(1-\omega^2u\big)^{a_2}}.\end{gather*}
The formula for $\phi(e_2)$ is similar. Now $\phi_T$ is well defined and satisfies~(a) and~(c). Property~(b) is clear for the case that all $a_0$, $a_1$, $a_2$ are integers. Since $\mathbb{Z}$ is dense in~$\mathbb{Z}_p$, (b) holds for all $a_0$, $a_1$, $a_2$. By construction~(d) holds.

The $H$-invariant vector space $M=C(u)e_1+C(u)e_2+C(u)e_1^{-1}e_2^{-1}$ is a three-dimensional iterative differential module over $C(u)$. The formulas for $\frac{\phi e_i}{e_i},\ i=1,2$ show that the singular points are $u\in \big\{0,\infty, 1,\omega, \omega^2\big\}$.

For generic $a_0$, $a_1$, $a_2$ with $a_0+a_1+a_2=0$ (a sufficient condition is that $a_1$, $a_2$, $1$ are linearly independent over~$\mathbb{Q}$) its Picard--Vessiot field is
$K$. Let $N$ be $M^{\langle \rho \rangle}:=\{m\in M\, | \, \rho(m)=m\}$. This is an iterative module over~$C(z)$ of dimension three. It has again $K$ as Picard--Vessiot field and its differential Galois group is~$H$. The singularities are the points $u$ with $u=1,\omega, \omega^2, 0, \infty$. Over $C(z)$ the singular points are $0$, $1$, $\infty$.

{\bf The general case.} Suppose that $(X,S)$ and $H$, a semi-direct product of a torus $\mathbb{T}$ and a~finite group~$F$ with order prime to $p$ are given. The prime-to-$p$ fundamental group of $X\setminus S$ is, by assumption, large enough for the existence of a surjective homomorphism to~$F$. There results a Galois covering $\big(\tilde{X},\tilde{S}\big)\rightarrow (X,S)$ with group $F$. The function fields of $X$ and $\tilde{X}$ are called \mbox{$K_1\subset K_2$}. Let $\dim \mathbb{T}=d$ and consider the purely transcendental extension $K=K_2(e_1,\dots ,e_d)$ of $K_2$. We interpret $e_1,\dots ,e_d$ as a basis of the characters of~$\mathbb{T}$ and let~$F$ act on the multiplicative group $e_1^\mathbb{Z} \cdots e_d^\mathbb{Z}$ as given by the semi-direct product. This defines an action of $F$ on~$K$. Furthermore $\mathbb{T}$ acts trivially on $K_2$ and acts on $e_1,\dots, e_d$ according to their identification with characters. In this way one obtains an action of~$H$ on $K$. The field $K$ with the group $H$ will be the Picard--Vessiot field of a, to be constructed, iterative module.

The next step is to define a suitable $C$-linear homomorphism $\phi_T\colon K\rightarrow K[[T]]$ which provides~$K$ with higher derivations. For this we need enough independent invertible elements $f_1,\dots ,f_n$ (modulo $C^*$) in the ring of regular functions of $\tilde{X}\setminus \tilde{S}$. By the assumption on the rank of $O(X\setminus S)^*/C^*$, these elements exist.

Then the transcendental elements $e_i$ are seen as symbols $f_1^{a_{i,1}}\cdots f_n^{a_{i,n}}$ where the $a_{i,j}\in \mathbb{Z}_p$ are chosen such that $F$ acts on the symbols
in the same way that $F$ acts on the characters $\{e_i\}$.

Let the finite, $F$-invariant set $E\subset \mathbb{T}$ generate~$\mathbb{T}$. Then $M=\oplus _{e\in E}K_2e$ is an iterative differential submodule of $K$. For a generic choice of the $a_{i,j}$, its Picard--Vessiot field is~$K$. The singularities of $M$ are contained in~$\tilde{S}$. Then $N=M^F=\{m\in M\, | \, \rho(m)=m \ \text{for all} \ \rho \in F\}$,
is an iterative differential module over $K_1$ and has the required properties. \end{proof}

\begin{Remarks}\label{2.3} If $C=\overline{\mathbb{F}_p}$, then $O(X\setminus S)^*/C^*$ has rank $-1+\#S$. Indeed, we may suppose that $X$ has genus $g\geq 1$. Consider the map $O(X\setminus S)^*\rightarrow \operatorname{Div}^0(S)$ which sends a function to its divisor on $X$ with support in $S$. The kernel is $C^*$ and the image $\operatorname{Prin}(S)$ consists of the principle divisors with support in~$S$. The statement follows from $\operatorname{Div}^0(S)/\operatorname{Prin}(S)\subset \operatorname{Jac}(X)(C)$ and the latter group is a torsion group.

For $C\neq \overline{\mathbb{F}_p}$, the answer to Question~\ref{Question1} for~$(X,S)$ will in general not only depend on~$\# S$ but also on the position of the points $S$. Indeed, let $E$ be any elliptic curve and $S=\{q_1,q_2\}$. Then $O(E\setminus S)^*/C^*\neq \{1\}$ if and only if there is a rational function $f$ on $E$ with divisor $n([q_1]-[q_2])$. This is equivalent to $q_2=q_1\oplus t$ (addition on $E$) where $t$ is a torsion element of~$E$. The group~$E(C)$ is not a torsion group since $C\neq \overline{\mathbb{F}_p}$.\end{Remarks}

\subsection{Question~\ref{Question2} for non-connected groups}\label{section2.3}

We will use a nice and important result of \cite[Theorem~21]{Sa}, namely the description of the stratified fundamental group of an abelian variety $A$ over $C$. We adopt the following notion and notation. Let $\mathbb{X}$ be any commutative group. Then $\operatorname{Diag}(\mathbb{X})$ denotes the commutative affine group scheme
$\operatorname{Hom}(\mathbb{X},C)$. Thus the affine ring of $\operatorname{Diag}(\mathbb{X})$ is the group algebra~$C[\mathbb{X}]$. In particular $\operatorname{Diag}\big(\mathbb{Z}^d\big)$ is the $d$-dimensional torus over $C$. The statement is:

\begin{Theorem}[\cite{Sa}]\label{2.4} The stratified fundamental group of the abelian variety $A$ is the product of~the $p$-adic Tate module $T_p(A)$ and $\operatorname{Diag}(\mathbb{X})$. Here $\mathbb{X}$ is the projective limit of $A^*\overset{[p]}{\leftarrow}A^*\overset{[p]}{\leftarrow}A^*\overset{[p]}{\leftarrow}\cdots$ and~$A^*$ denotes the group of the $C$-points of the dual abelian variety, seen as an ``abstract group''. Finally, $[p]$ denotes the multiplication by $p$.\end{Theorem}

Using this, we give an alternative proof of \cite[Theorem~7.1(3)]{Ma-vdP1}, which can also be formulated as

\begin{Proposition}\label{2.5} \samepage Any torus can be realized for $(X,S)$ with $X$ of genus $g$, except for the cases:
\begin{enumerate}\itemsep=0pt
\item[$(a)$] $g=0$ and $\# S \leq 1$,
\item[$(b)$] $g>0$, $\# S\leq 1$, $T_p(\operatorname{Jac}(X))=0$ and $C=\overline{\mathbb{F}_p}$.
\end{enumerate}
 \end{Proposition}

\begin{proof} (a) For $(\mathbb{P}^1,\{0,\infty \})$, the example $E(\alpha)$, $\alpha \in \mathbb{Z}_p$, $\alpha \not \in \mathbb{Q}$ (see Section~\ref{section3} for details) has differential Galois group~$\mathbb{G}_m$. It now suffices to show that~$\mathbb{G}_m$ cannot be realized for $\big(\mathbb{P}^1,\{\infty \}\big)$.

The equivalence between the categories of stratification and of projective systems (see for instance \cite[Theorem~8]{Sa}) applied to $\big(\mathbb{P}^1,\{\infty \}\big)$ translates a~1-dimensional stratification into a~projective system $C[t]e_0\supset C[t]^pe_1\supset C[t]^{p^2}e_2\supset \cdots$ and for all $n$, one has $e_{n+1}=a_n e_n$ with~$a_n$ invertible in $C[t]^{p^n}$. Since all $a_n\in C^*$, one may suppose that all $a_n=1$. Then the system is trivial and the differential Galois group is~$\{1\}$.

(b) Let $X$ have genus $g\geq 1$. Suppose $\mathbb{G}_m$ is not realizable for~$(X,S)$, then the same holds for $(X,\varnothing)$. The abelianized stratified fundamental group of~$X$ equals the stratified fundamental group of the Jacobian variety $\operatorname{Jac}(X)$ of $X$. By Theorem~\ref{2.4}, $\mathbb{G}_m$ cannot be realized if and only if $\mathbb{X}$ is a torsion group. The latter is equivalent to $T_p(\operatorname{Jac}(X))=0$ and $C=\overline{\mathbb{F}_p}$.

Consider now the case $\#S\geq 2$. Then $O(X\setminus S)$ contains a non-constant invertible element $t$ since we may assume that $\operatorname{Jac}(X)(C)$ is a torsion group. The pullback of $E(\alpha )$ on~$\big(\mathbb{P}^1,\{0,\infty \}\big)$ under the morphism $t\colon X\rightarrow \mathbb{P}^1$ produces a stratification on~$(X,S)$ with differential Galois group~$\mathbb{G}_m$.

Finally, consider a point $s\in X$. If $\mathbb{G}_m$ is realizable for $(X,\{s\})$, then the same holds for $(X,\varnothing)$. Indeed, a 1-dimensional projective system for the case $(X,\{s\})$ extends to a 1-dimen\-sional system for $(X,\varnothing)$. \end{proof}

Now we explain counterexamples, i.e., the negative answers for Question~\ref{Question2}, obtained by A.~Maurischat \cite[Theorem~9.1]{R}.
Let $G=\mathbb{T}\rtimes \mathbb{Z}/p\mathbb{Z}$, where $\mathbb{T}$ is the torus $\{(t_1,\dots,t_p)\in (C^*)^p\, | \, t_1\cdots t_p=1\}$ and for a generator $\sigma$ of $\mathbb{Z}/p\mathbb{Z}$ one has $\sigma (t_1,\dots ,t_p)\sigma^{-1}$ is the cyclic permutation $(t_2,\dots ,t_p,t_1)$ of $(t_1,\dots ,t_p)$. One easily verifies that $p(G)=G$.

\begin{Proposition} \label{2.6} $G$ cannot be realized for $\big(\mathbb{P}^1,\{\infty \}\big)$ if $C$ is the field $\overline{\mathbb{F}_p}$.
\end{Proposition}
\begin{proof} Suppose that $G$ can be realized. Then $G/G^o$ defines a cyclic covering $h\colon Z\rightarrow \mathbb{P}^1$ of degree $p$, ramified only above $\infty$. Then $G^o=\mathbb{T}$ and then also~$\mathbb{G}_m$ are realized for $\big(Z,h^{-1}(\infty)\big)$. The equation of $Z$ has the form $s^p-s=f(z)\in C[z]$. If the genus of $Z$ is zero, then this contradicts Proposition~\ref{2.5}(a). If the genus of $Z$ is $>0$, then, according to the Deuring--Shafarevich formula, the Jacobian variety $\operatorname{Jac}(Z)$ has $p$-rank zero, see \cite[Theorem~1.1]{Sh}. Now we obtain a contradiction with Proposition~\ref{2.5}(b). \end{proof}

\begin{Corollary}\label{2.7} Let $G$ denote the group of Proposition {\rm \ref{2.6}}.
\begin{enumerate}\itemsep=0pt
\item[$(a)$] $G$ is realizable for $\big(\mathbb{P}^1,\{\infty \}\big)$ if $C\neq \overline{\mathbb{F}_p}$ and
\item[$(b)$] for $\big(\mathbb{P}^1,\{0, \infty \}\big)$ if $C=\overline{\mathbb{F}_p}$.
\end{enumerate}
 \end{Corollary}
\begin{proof} (a) Consider a $p$-cyclic Galois covering $h\colon Z\rightarrow \mathbb{P}^1$, only ramified above $\infty$ and with genus $>0$. Let $\xi \in \operatorname{Jac}(Z)(C)$ be an element of infinite
order, which exists since $C\neq \overline{\mathbb{F}_p}$. Let $\xi$ be represented by a divisor
$D$ of degree 0 on $Z$. Then $D+\sigma (D)+\cdots +\sigma^{p-1}(D)$ is an invariant divisor
on $Z$ and is trivial since $Z/\langle \sigma \rangle=\mathbb{P}^1$. Thus
$\xi+\sigma (\xi)+\cdots +\sigma^{p-1}(\xi)$ is the zero element of $\operatorname{Jac}(Z)$.
 Since the polynomial $1+T+\cdots +T^{p-1}$ is irreducible over $\mathbb{Q}$ and $\xi$
 has infinite order we conclude that the `only' integral relation between the elements $\xi, \sigma (\xi),\dots ,\sigma^{p-1}(\xi)$ is $\xi+\sigma (\xi)+\cdots +\sigma^{p-1}(\xi)=0$.

The character group $A$ of $\mathbb{T}$ is $\{(a_1,\dots ,a_p) \in \mathbb{Z}^p\, | \, a_1+\cdots +a_p=0\}$. The morphism $m\colon (a_1,\dots ,a_p)\in A\mapsto \sum\limits_{i=1}^pa_i\sigma^{i-1}(\xi)\in \mathbb{X}$ yields a surjective, equivariant homomorphism $\operatorname{Diag}(\mathbb{X})\allowbreak \rightarrow \mathbb{T}$. This corresponds to a stratified bundle~$M$ on $Z$ having differential Galois group~$\mathbb{T}$. Since~$m$ is equivariant for the action of $\sigma$, the stratification $M$ is $\langle \sigma \rangle$-equivariant one has that $M=h^*N$ for a stratification~$N$ on $\mathbb{P}^1$ with only a singularity at~$\infty$. By construction, $G$ is the differential Galois group of~$N$ (more details in Proposition~\ref{2.8}).

(b)~For convenience we suppose $p=2$. Consider the cyclic covering $m\colon Z\rightarrow \mathbb{P}^1$ with group $\{1,\rho \}$ in the proof of Proposition~\ref{2.6}. Write
 $m^{-1}\{0\}=\{0_1,0_2\}$. The divisor class of $[0_1]-[0_2]$ has finite order $\ell \geq 1$. Take an element $t\in C(Z)$ with $\operatorname{div}(t)=\ell ([0_1]-[0_2])$. We normalize $t$ such that $\rho (t)=t^{-1}$. Let $R=O\big(Z\setminus \big\{0_1,0_2,m^{-1}(\infty )\big\}\big)$. We now produce a~projective system which defines a stratification on $Z\setminus \big\{0_1,0_2,m^{-1}(\infty )\big\}$ with differential Galois group $\mathbb{G}_m$ by the method of \cite[Theorem~7.3, Lemma~7.4]{Ma-vdP1}.

 This projective system reads $Re_0\supset R^pe_1\supset R^{p^2}e_2\supset \cdots$,
 where, for all $n\geq 0$, one has $e_{n+1}=a_ne_n$, $a_n= \big(t^{p^n}\big)^{b_n}$ for certain $b_n\in \{0,1\}$. We require
 that infinitely often $b_n=1$ and that there are ``large gaps'' where $b_n=0$. These conditions ensure that
 the corresponding 1-dimensional differential module $M$ over $C(Z)$ has differential Galois group $\mathbb{G}_m$.
 Moreover $\rho$ applied to $M$ produces the dual (inverse) $M^*$ of $M$. It follows that
 $M\oplus M^*$ is $\langle \sigma \rangle$-equivariant. Then there is a stratification $N$
 on $\mathbb{P}^1$ with $h^*N=M\oplus M^*$. The stratification $N$ has the required
 properties (more details in Proposition~\ref{2.8}).
 \end{proof}

{\it We describe, correct and supplement methods and results of} \cite[Section~8]{Ma-vdP2}, \cite[Section~8]{Ma} {\it and}~\cite{Er} {\it concerning the non-connected case for Question~{\rm \ref{Question2}}.}

Question~\ref{Question2} has a translation into the following {\it embedding problem}: Suppose that $G/p(G)$ is realizable for $(X,S)$. We denote the finite group $G/G^o$ by $K$. The canonical map $G\rightarrow K$ induces a surjective map $G/p(G)\rightarrow K/p(K)$. Then also $K/p(K)$ is realizable for $(X,S)$. Since Question~\ref{Question2} has a positive answer for finite groups, $K$ is realizable for $(X,S)$. This leads to an embedding problem for the exact sequence $1\rightarrow G^o\rightarrow G\rightarrow K\rightarrow 1$. Now we describe how to obtain a proper solution. As explained in the beginning of Section~\ref{section2.2} we may and will restrict ourselves to the case that $G\subset {\rm GL}(V)$ is a semi-direct product of $K$ and $G^o$.

A finite Galois covering $h\colon \big(\tilde{X},\tilde{S}\big)\rightarrow (X,S)$ with group~$K$ and where $\tilde{S}:=h^{-1}(S)$, is given. The embedding problem has a proper solution if one can produce on~$\tilde{X}\setminus \tilde{S}$ a stratification~$M$, with differential Galois group $G^o$ and equivariant for the action of the group $K$ on the covering. {\it In Proposition~{\rm \ref{2.8}} we will show that the existence of $M$ implies that $G$ is realizable for $(X,S)$.}

{\it We define the $K$-equivariance of $M$ as follows}. For every $k\in K$ a $k$-semi-linear isomorphism $\Phi (k)\colon M\rightarrow M$ is given such that
$\Phi (k_1k_2)= \Phi(k_1)\circ \Phi(k_2)$ for all $k_1,k_2\in K$. This condition on~$\Phi (k)$ means that~$\Phi(k)$ is an automorphism of the abelian sheaf~$M$, commutes with the action of the differential operators $h^{-1}\mathcal{D}_{X}$ and $\Phi (k)(\lambda \cdot m)= k(\lambda) \cdot \Phi(k)(m)$ for sections $\lambda, m$ of the sheaves $O_{\tilde{X}}$ and~$M$. The $K$-equivariance for~$M$ are in fact {\it descent data} and imply that~$M=h^*N$ for a unique stratification~$N$ on~$X\setminus S$.

We note that there is {\it an alternative definition for $K$-equivariance of~$M$}. For every $k\in K$ one defines a twist $^kM$ of $M$ by~$^kM$ and~$M$ are equal as abelian sheaves and for sections~$\lambda$ and~$m$ of~$O_{\tilde{X}}$ and~$M$, the new multiplication $\lambda \star m$ on~$^kM$ is defined as $k^{-1}(\lambda )\cdot m$. Furthermore isomorphisms $\phi (k)\colon ^kM\rightarrow M$ are given such that $\phi(k_1k_2)=k_1\circ \ ^{k_1}(\phi (k_2))$ for all~$k_1$,~$k_2$.

 Let $V$ be the above vector space over $C$ of dimension $d$ with $G\subset {\rm GL}(V)$. One constructs the stratification $M$ on $\tilde{X}\setminus \tilde{S}$ by producing a~projective system
 \begin{gather*} R\otimes_CV\overset{D_0}{\leftarrow} R^p\otimes_CV\overset{D_1}{\leftarrow} R^{p^2}\otimes_CV \overset{D_2}{\leftarrow} \cdots , \end{gather*}
 where $R=O\big(\tilde{X} \setminus \tilde{S}\big)$, $D_\ell \in G^o\big(R^{p^\ell}\big)$ for all $\ell$. The Galois group $K$ acts on ${\rm GL}(R\otimes _CV)$ as follows. The choice of a basis of $V$ over $C$ gives an identification of the last group with ${\rm GL}_d(R)$.

 Any $k\in K$ acts on a matrix in ${\rm GL}_d(R)$ by its Galois action on the entries of the matrix. The equivariance condition is $k(D_\ell)=k\cdot D_\ell \cdot k^{-1}$ for all $k\in K$ and all $\ell$, where $k$, $k^{-1}$ on the right hand side of the equality are seen as elements of ${\rm GL}(V)\subset {\rm GL}(R\otimes _CV)$.

 One assumes that the following conditions on the projective system are satisfied (compare \cite[Lemma 7.4]{Ma-vdP1} for details):
\begin{enumerate}\itemsep=0pt
\item[(i)] For every $m$ the group $G^o$ is topologically generated by $\{D_\ell \, |\, \ell \geq m\}$.
\item[(ii)] The degrees of the $D_\ell$ are bounded.
\item[(iii)] There are large gaps.
\end{enumerate}

 \begin{Proposition}\label{2.8} We use the above notation. Suppose that the matrices $D_\ell$ satisfy the equivariance conditions and that $(i)$--$(iii)$ hold. Then~$G$ is the differential Galois group for a stratification on $X\setminus S$ defined by these data.
\end{Proposition}

\begin{proof} (1) Suppose that a $K$-equivariant stratification $M$ on $\tilde{X}\setminus \tilde{S}$ with group $G^o$ is given. Let $P\supset C\big(\tilde{X}\big)$ a Picard--Vessiot field for $M$ over $C\big(\tilde{X}\big)$. From the equivariance of $M$ it follows that the action of the group $K$ on $C\big(\tilde{X}\big)/C(X)$ extends to an action of $P/C(X)$ and that the group of differential automorphisms of $P/C(X)$ is~$G$. An easy way to verify this is to write~$M$ as a~system of matrix equations $\big\{\partial^{(n)}y=A_ny\big\}$ where the matrices $A_n$ have their entries in $O\big(\tilde{X}\setminus \tilde{S}\big)$. Let $F=(f_{i,j})$ be a~fundamental matrix. The field $P$ is generated over $C\big(\tilde{X}\big)$ by the~$f_{i,j}$. The equivariance condition $k(A_n)= k\cdot A_n\cdot k^{-1}$ implies that $k\cdot F\cdot k^{-1}$ is a~fundamental matrix for the system $\big\{\partial ^{(n)}y=k(A_n)y\big\}$. The proposed action of any $k\in K$ on~$P$ is by defining $k(f_{i,j})\in P$ such that $(k(f_{i,j}))=k\cdot (f_{i,j})\cdot k^{-1}$. It is easy to verify the required properties.

Let $N$ denote the stratification on $X\setminus S$ with $h^*N=M$ and let~$N_0$ be $O\big(\tilde{X}\setminus \tilde{S}\big)$ seen as stratification on $X\setminus S$. Then $P$ is the Picard--Vessiot field for $N\oplus N_0$. Indeed, a Picard--Vessiot field~$\tilde{P}$ for $N\oplus N_0$ is contained in $P$. It contains $C\big(\tilde{X}\big)$ since this is the Picard--Vessiot field for~$N_0$. Furthermore, $\tilde{P}$ contains a Picard--Vessiot field for $M=h^*N$ and thus $\tilde{P}=P$.

(2) {\it We give a sketch of the proof} that the properties (i)--(iii) imply that $G^o$ is the differential Galois group for the given projective system over $\tilde{X}\setminus \tilde{S}$. It is easy to see that \cite[Proposition~5.3]{Ma-vdP1} is valid in this more general setting. This implies that the differential Galois group $H$ is contained in~$G^o$.

If $H$ is a proper subgroup, then there is a construction of linear algebra $\operatorname{Csrt}(M)$ applied to~$M$, and a 1-dimensional object $L\subset \operatorname{Csrt}(M)$ which is invariant under $H$ but not under $G^o$. In order to obtain a contradiction one adapts \cite[Lemma~7.4]{Ma-vdP1} to the more general situation. This is done by replacing
$\mathbb{A}^1$ (or $\mathbb{A}^1\setminus \{0\}$) by $\tilde{X}\setminus \tilde{S}$ and choosing a~suitable ``degree'' function on the algebra $O\big(\tilde{X}\setminus \tilde{S}\big)$. \end{proof}

It is difficult to produce matrices $D_\ell\in G^o\big(R^{p^\ell}\big)$ with $k(D_\ell)=k\cdot D_\ell \cdot k^{-1}$. In~\cite{Ma} and~\cite{Ma-vdP2} one introduces for this purpose a~form $G^o_{\chi}$ of $C(X)\times _CG^o$ with $\chi\colon K=\operatorname{Gal}\big(C\big(\tilde{X}\big)/C(X)\big)\rightarrow \operatorname{Aut}\big(C\big(\tilde{X}\big)\otimes _CG^o\big)$ defined by: $k\in K$ is send to the map $D\mapsto k^{-1}\circ k(D)\circ k$ (see \cite[Section~11.3]{Sp} for definitions and details). For each $\ell$ there is a~form of $C(X)^{p^\ell}\times _CG^o$, similarly defined, which we denote by $G^o_{\chi , \ell}$. The required $D_\ell $ are elements of $G^o_{\chi, \ell}\big(C(X)^{p^\ell}\big)$.

Proposition \ref{2.6} is a counterexample to \cite[Proposition~8.7, Theorems~8.8 and 8.11]{Ma-vdP2}. The sketch of the proof of \cite[Proposition~8.7]{Ma-vdP2} uses $G^o_{\chi}$ which is defined over the field~$C(X)$ but, a~priori, not over $O(X\setminus S)$. Furthermore it has the correct statement that a torus over a field like~$C(t)$ is topologically generated by one element. However, for instance in the case $\big(\mathbb{P}^1,\{\infty \}\big)$ one has to consider instead $\mathbb{G}_m(C[t])=C^*$. This is a torsion group if $C=\overline{\mathbb{F}_p}$ and is not topologically finitely generated.

The methods of \cite[Section~8]{Ma} and \cite[Section~8]{Ma-vdP2} prove the {\it weaker versions of} \cite[Proposition~8.7, Theorems~8.8 and 8.11]{Ma-vdP2} where the singular locus of the to be constructed stratification is not specified. Indeed, one can verify that the proofs use finitely many rational functions which can have poles outside~$S$. Thus for a larger set $S^+\supset S$ (depending on the case) the constructions and proofs work. This has the consequence, for example:
 \begin{Proposition}\label{2.9} Any linear algebraic group $G$ over $C$ is realizable over $\big(\mathbb{P}^1,S\big)$ for some~$S$.
\end{Proposition}
\begin{proof} $G/G^o$ can be realized for $\big(\mathbb{P}^1,S\big)$ if $\#S$ is large enough. The embedding problem has a~proper solution after possibly enlarging $S$.
\end{proof}

A different embedding problem has been studied by S.~Ernst~\cite{Er}. Consider a linear algebraic group $G$ and a~Zariski closed (reduced) normal subgroup $N$. Suppose that~$G/N$ can be realized as differential Galois group and therefore is the automorphism group of a Picard--Vessiot extension~$E$ of an algebraic function field~$C(X)$. Then there is an extension of Picard--Vessiot fields $E'\supset E$ over $C(X)$ such that the differential Galois group of~$E'$ is~$G$.

For the case that $G/N$ is finite one recovers \cite[Proposition~8.7]{Ma-vdP2} without specifying the singular locus of the iterative differential module over $C(X)$.

\begin{Remarks} The ``differential Abhyankar conjecture'' is still open.
\begin{enumerate}\itemsep=0pt
\item[(1)] An email message of June 11, 2019 by B.H.~Matzat confirms that one has to interpret Proposition~8.7 and Theorem~8.8 of~\cite{Ma-vdP2} in such a way that increasing the set $S$ is allowed. Furthermore, it is proposed to replace \cite[Theorem~8.11]{Ma-vdP2} by the weaker statement:
 \begin{enumerate}\itemsep=0pt
 \item[]{\it $G$ is realizable for $\big(\mathbb{P}^1,\{\infty \}\big)$ if $p(G^o)=G^o$ and $p(G)=G$} (there is no proof yet).
 \end{enumerate}
\item[(2)] In case $C=\overline{\mathbb{F}_p}$ we do not know which non-connected groups $G$ with $p(G)=G$ are realizable for $\big(\mathbb{P}^1,\{\infty \}\big)$. The answer depends on the possible $p$-ranks of Galois coverings $Z\rightarrow \mathbb{P}^1$, only ramified above~$\infty$.
\item[(3)] In case $C\neq \overline{\mathbb{F}_p}$ it seems likely that Question~\ref{Question2} has a positive answer for all cases $(X,S)$.
\end{enumerate}
\end{Remarks}

\section{Regular singular stratifications}\label{section3}

In the {\it local formal case} one considers the field $C((t))$ provided with the higher derivations $\big\{\partial_t ^{(n)}\big\}$ defined by $\partial _t^{(n)}\big(t^m\big)={m\choose n}t^{m-n}$. A vector space $M$ of finite dimension over $C((t))$, provided with a stratification, say in the form of operators $\partial_M^{(n)}$ acting on $M$, is called a {\it regular singular stratification} if there is a $C[[t]]$-lattice $\Lambda \subset M$, which is invariant under all $t^n\partial_M^{(n)}$.

A {\it tame field extension} $F\supset C((t))$ is a finite field extension of degree $m$ not divisible by $p$. It is well known that $F=C((s))$ with $s^m=t$. The higher derivations of $C((t))$ extend in a unique way to explicit higher derivations on~$F$. Furthermore, for any regular singular stratification~$M$ over $C((t))$ the stratification $F\otimes _{C((t))}M$ is easily seen to be regular singular.

We consider a {\it stratification $M$ on $X$ with singularities in $S$}, i.e., $M$ is a vector bundle on $X$ and the restriction of $M$ to $X\setminus S$ is a stratification. Consider $s\in S$ with local parameter $t$. Let $\widehat{M}_s$ be the completion of the stalk of $M$ at $s$. This is a free finitely generated $\widehat{O_{X,s}}=C[[t]]$-module and $\widehat{M_s}\otimes _{C[[t]]}C((t))$ is a stratified module over~$C((t))$. If~$s$ is a regular singular point, then $\widehat{M}_s\otimes_{C[[t]]} C((t))$ is also regular singular. We will show that the converse is also true.

We note that the examples in \cite[Sections~4.2 and~7]{Ma-vdP1} are not counterexamples to Lemma~\ref{3.1} since they concern stratifications over the field~$C(X)$ that do not come from (regular) stratifications on some $X\setminus S$.

\begin{Lemma}\label{3.1} Let $X$, $M$, $S$, $s$, $t$ be as above. Suppose that $\widehat{M}_s\otimes _{C[[t]]}C((t))$ is regular singular. Then $s$ is a regular singularity for $M$.
 \end{Lemma}
\begin{proof} We use on $X$ and $M$ the higher derivations $\big\{\partial_t^{(n)}\big\}$ and the corresponding $\big\{\partial_M^{(n)}\big\}$. As explained in Section~\ref{section2.1}, the stalk $M_\eta$ of~$M$ at the generic point of~$X$ is an iterative differential module for the action of the $\partial_M^{(n)}$. Put $d=\dim M_\eta$. It has the property that for any basis $e_1,\dots ,e_d$ of $M_\eta$ there is a finite subset $T$ of $X$ such that $O(X\setminus T)e_1\oplus \cdots \oplus O(X\setminus T)e_d$ is invariant under all $\partial_M^{(n)}$. Indeed, this follows from the assumption that $M_\eta$ comes from a~(regular) stratification on some open subset of~$X$.

By assumption $M_\eta\otimes _{C(X)}C((t))=\widehat{M}_s\otimes _{C[[t]]}C((t))$ has a $C[[t]]$-lattice which is invariant under all $t^n\partial_M^{(n)}$. Any lattice is generated by a basis of the $C((t))$-vector space $M_\eta\otimes _{C(X)}C((t))$. The elements of $M_\eta$ are dense in $M_\eta\otimes _{C(X)}C((t))$ and so we may suppose that the lattice is generated by a basis $e_1,\dots ,e_d$ of $M_\eta$. For a small enough affine neighbourhood~$U$ of~$s$, the maximal ideal in~$O(U)$ of the point~$s$ is generated by $t$ and the stratified module $N:=O(U)e_1+\cdots +O(U)e_d$ has the property that the only singularity is $s$. Thus for every $n$ there is a smallest integer $i(n)\geq 0$ such that $t^{i(n)}\partial _M^{(n)}N\subset N$. Taking the completion at~$s$ (or what is the same, taking the tensor product over $O(U)$ with $\widehat{O_{X,s}}$ ) does not change the numbers $i(n)$. By assumption $i(n)\leq n$ and so $s$ is a regular singularity. \end{proof}

We note that Lemma~\ref{3.1} is also present in \cite[Proposition 3.2.3]{K}.

Consider, as before, the field $C((t))$ with the higher derivations $\big\{\partial_t ^{(n)}\big\}$. Let~$\alpha$ be a $p$-adic integer. For any integer $n\geq 0$, the binomial coefficient ${\alpha \choose n}$ is also in $\mathbb{Z}_p$. Its reduction modulo $p$ to an element in $\mathbb{F}_p$ is denoted by the same symbol. For any $\alpha \in \mathbb{Z}_p$ we define the 1-dimensional regular singular stratified $C((t))$-vector space $E(\alpha )=C((t))e$ by $\partial ^{(n)}e={\alpha \choose n }t^{-n}e$ for all $n\geq 0$. Furthermore the stratifications $E(\alpha )$ and $E(\beta)$ are isomorphic if and only if $\alpha -\beta \in \mathbb{Z}$. Now~$C(t)e$ with the same formulas defines on $\mathbb{P}^1\setminus \{0,\infty \}$ a~regular singular stratification which will also be denoted by~$E(\alpha )$.

According to \cite[Theorem 3.3]{Gie} and \cite[Proposition 6.1]{Ma-vdP1}, any finite-dimensional regular singular stratified vector space~$M$ over~$C((t))$ is a direct sum $E(\alpha _1)\oplus \cdots \oplus E(\alpha _d)$. The elements $\alpha_1,\dots ,\alpha _d\in \mathbb{Z}_p$ are called {\it the local exponents}. Their images in $\mathbb{Z}_p/\mathbb{Z}$ are uniquely determined by~$M$. The differential Galois group of~$M$ is the group $\operatorname{Diag}(\mathbb{X})$ where $\mathbb{X}$ is the subgroup of~$\mathbb{Z}_p/\mathbb{Z}$ generated by the images of~$\alpha_1,\dots ,\alpha _d$.

A Galois (\'etale) covering $\tilde{X}\setminus \tilde{S}\rightarrow X\setminus S$ produces a stratification on~$X$. This is the finitely generated projective $O(X\setminus S)$-module $O\big(\tilde{X}\setminus \tilde{S}\big)$ provided with the left action of $\mathcal{D}(X\setminus S)$ which uniquely extends the left action on $O(X\setminus S)$ itself. A~point $s\in S$ is regular singular for the stratification if and only if the ramification is ``tame'' (i.e., the ramification index is prime to $p$). This is an easy special case of a theorem of L.~Kindler~\cite{Ki}.

 \subsection[Regular singular stratifications on $\mathbb{P}^1\setminus S$]{Regular singular stratifications on $\boldsymbol{\mathbb{P}^1\setminus S}$}\label{section3.1}

 \begin{Proposition}[Gieseker]\label{3.2} {\rm Regular singular stratifications on $\mathbb{P}^1\setminus \{0,\infty \} $}.
 \begin{enumerate}\itemsep=0pt
\item[$(1)$] $E(\alpha)$ is isomorphic to $E(\beta)$ if and only if $\alpha -\beta \in \mathbb{Z}$.
\item[$(2)$] Any regular singular stratification $($of rank~$m)$ on $\mathbb{P}^1\setminus \{0,\infty \}$ is a direct sum $\oplus_{i=1}^m E(\alpha _i)$ $($all $\alpha _i\in \mathbb{Z}_p)$ of $1$-dimensional stratifications. The images of the $\alpha_i$ in $\mathbb{Z}_p/\mathbb{Z}$ are uniquely determined by the stratification.
\item[$(3)$]The regular singular stratified fundamental group $\pi^{\rm str,rs}\big(\mathbb{P}^1\setminus \{0,\infty \},x_0\big)$ is equal to $\operatorname{Diag}(\mathbb{X})$ with
$\mathbb{X}=\mathbb{Z}_p/\mathbb{Z}$.
\end{enumerate}
 \end{Proposition}
 We refer to \cite[Section~4]{Gie} for a proof of (1) and (2) of Proposition~\ref{3.2}. Part (3) follows easily.

\begin{Remarks}\label{3.3}\quad
 \begin{enumerate}\itemsep=0pt
\item[(1)] From Proposition~\ref{3.2} it follows that any regular singular stratification on $\mathbb{P}^1\setminus S$ with $\#S\leq 1$ is trivial.
\item[(2)] The case $\# S=3$, i.e., the case $\mathbb{P}^1\setminus \{0,1,\infty \}$ is rather different. We have only a $p$-adic method to produce regular singular stratifications on this affine curve. This gives explicit results for the analogue of the classical hypergeometric equations. We will treat this case in Section~\ref{section3.2}.
\item[(3)] For $\# S\geq 4$ a new method for constructing regular singular stratifications, using {\it Mumford groups} will be treated in Section~\ref{section5}. Similar rigid methods produce stratifications on Mumford curves, see Section~\ref{section4}.
\item[(4)] For $X$ of genus $\geq 2$, constructions of stratified bundles will be given, using Schottky groups and Mumford curves, in Section~\ref{section4}.
\end{enumerate}
\end{Remarks}

\begin{Proposition}\label{3.4} Let $G$ be the differential Galois group of a regular singular stratification $M$ on~$\mathbb{P}^1\setminus \{a_1,\dots ,a_r,\infty \}$. Let $t_i$ denote a local coordinate at $a_i$ and let $G_i$ be the local differential Galois group for $ C((t_i))\otimes M_{a_i}$.
 \begin{enumerate}\itemsep=0pt
\item[$(1)$] There is a natural embedding $G_i\subset G$, unique up to conjugation in $G$.
\item[$(2)$] $G$ is topologically generated by all conjugates of the $G_1,\dots ,G_r$.
\end{enumerate}
\end{Proposition}

The differential Galois group of a regular singular differential equation on the {\it complex} puncture projective line is generated by the local monodromy groups (or their Zariski closures). Proposition~\ref{3.4} is a characteristic $p>0$ version of this.
\begin{proof} (1) Let $K\supset C(z):=C\big(\mathbb{P}^1\big)$ be a Picard--Vessiot field for~$M$. Let $K_i\supset C((t_i))$ be a Picard--Vessiot field for $ C((t_i))\otimes M_{a_i}$ and write $F_i$ for a fundamental matrix with entries in $K_i$ with respect to a basis of~$M$. The iterative subfield of $K_i$ generated over $C(z)$ by the entries of the $F_i$ and $\frac{1}{\det F_i}$ is a Picard--Vessiot field for~$M$ over~$C(z)$. By uniqueness of the Picard--Vessiot field there is a morphism of iterative differential fields $K\rightarrow K_i$ (extending the inclusion $C\big(\mathbb{P}^1\big)\subset C((t_i))$) which is unique up to differential automorphisms of~$K$ over~$C(z)$. From this~(1) follows.

(2) We have to show the following: If $N\subset G$ is a closed normal subgroup containing all~$G_i$, then $N=G$. Let $\{\{M\}\}$ denote Tannakian category generated by $M$. This category is equivalent to the one of the representations of~$G$. The latter contains a faithful representation of~$G/N$. Thus there is an object $T\in \{\{M\}\}$ with group~$G/N$. Then $T$ is regular singular and since $G_i$ maps to $\{1\}$ in~$G/N$, the points $a_i$ are not singular. Thus $T$ has at most a regular singularity at $\infty$. From Remarks~\ref{3.3}(1) it follows that $T=\{1\}$. \end{proof}

\subsection[Hypergeometric stratifications on $\mathbb{P}^1\setminus \{0,1,\infty \}$]{Hypergeometric stratifications on $\boldsymbol{\mathbb{P}^1\setminus \{0,1,\infty \}}$}\label{section3.2}
\begin{Proposition}\label{3.5} For elements $\alpha _0,\alpha _1,\alpha _\infty \in \mathbb{Z}_p$ with $\alpha_0+\alpha_1+\alpha_\infty =0$ there is a unique $1$-dimensional regular singular stratification on $\mathbb{P}^1\setminus \{0,1,\infty \}$ with local exponents $\alpha_0$, $\alpha_1$, $\alpha_\infty$ at the points $0$, $1$, $\infty$.
 \end{Proposition}

\begin{proof} An easy way to obtain this stratification is to consider the symbolic expression $s:=z^{\alpha_0}(z-1)^{\alpha _1}$. Then, still working symbolically, $\partial^{(n)}s=\big\{{\alpha _0\choose n }z^{-n}+{\alpha_1\choose n}(z-1)^{-n} \big\}s$. Now, formally, define the free, rank one module $M(\alpha_0,\alpha _1)=C\big[z,\frac{1}{z(z-1)}\big]b$ with the action of the operators $\partial^{(n)}$ given by $\partial^{(n)}b=\big\{{\alpha _0\choose n }z^{-n}+{\alpha_1\choose n}(z-1)^{-n} \big\}b$ for all $n$. This definition makes sense for all choices of $\alpha _0,\alpha _1\in \mathbb{Z}_p$ and defines a stratification.

Indeed, if $\alpha_0, \alpha_1\in \mathbb{Z}$, then $z^{\alpha_0}(z-1)^{\alpha_1}\in C(z)$ and the $\partial^{(n)}$ obviously satisfy the properties of a stratification. For general $\alpha_0,\alpha _1\in \mathbb{Z}_p$ one verifies each formula needed for $M(\alpha _0,\alpha_1)$ being a stratification by approximating $\alpha _0,\alpha_1$ by elements in~$\mathbb{Z}$.

If $E$ is a 1-dimensional stratification with regular singularities and local exponents $\alpha_0$, $\alpha _1$, $\alpha_\infty$, then $E^*\otimes M(\alpha _0,\alpha _1)$ is trivial because it has trivial local exponents. Thus $E\cong M(\alpha_0,\alpha_1)$. \end{proof}

Proposition \ref{3.5} and its proof extend in an obvious way to the case $\mathbb{P}^1\setminus S$ for any finite~$S$. An interesting case concerns regular singular stratifications with finite differential Galois group~$G$. These are given by a surjective homomorphism $\pi_1^{\rm tame}\big(\mathbb{P}^1\setminus \{0,1,\infty \}\big)\rightarrow G$. The tame fundamental group of $\mathbb{P}^1\setminus \{0,1,\infty \}$ is not explicitly known. We are not aware of a conjecture concerning the structure of this group.

A related interesting case concerns the analogue of the classical complex hypergeometric differential equations. A regular singular stratification $M$ on $\mathbb{P}^1\setminus \{0,1,\infty \}$ is called {\it standard hy\-per\-geometric} if~$M$ has rank~2 and the local exponents at $0$, $1$, $\infty$ are $ 0$, $1-\gamma || 0$, $\gamma -\alpha -\beta || \alpha$,~$\beta $ and $\alpha , \beta , \gamma \in \mathbb{Z}_p$.

The following result is an improvement of \cite[Theorem 8.9]{Ma-vdP1}.
 \begin{Theorem} \label{3.6} The $p$-adic integers $\alpha$, $\beta$, $\gamma$ are written as standard expansions $\alpha =a_0+a_1p+a_2p^2+\cdots$ with all $a_i\in \{0,\dots ,p-1\}$. For $k\geq 1$, write $\alpha_k=a_0+a_1p+\cdots +a_{k-1}p^{k-1}$. Put $\beta =b_0+b_1p+b_2p^2+\cdots$, $\gamma =c_0+c_1p+c_2p^2+\cdots$, and define similarly $\beta_k$ and $\gamma_k$.

The second order $p$-adic hypergeometric differential equation $H$, i.e., $ z(z-1)F''+((\alpha +\beta +1)z-\gamma )F'+\alpha \beta F=0$, reduces to a standard hypergeometric stratification if $\max(\alpha_k,\beta_k)\geq \gamma_k$ holds for $k\gg 0$.

If $\alpha$, $\beta$, $\gamma $ have this property, then the local exponents of $H$ and of its reduction are the same $($up to a shift over integers$)$.
 \end{Theorem}
 \begin{proof} According to the proof of \cite[Theorem~8.9]{Ma-vdP1}, $H$ (in matrix form) reduces to a standard hypergeometric stratification if and only if the set of coefficients of the two formal or symbolic solutions
 \begin{gather*}
 F_1=\sum _n \frac{(\alpha)_n(\beta)_n}{(\gamma)_nn!}z^n \qquad \text{and} \qquad F_2=z^{1-\gamma}(1-z)^{\gamma -\alpha -\beta} \sum _n\frac{(1-\alpha)_n(1-\beta )_n}{(2-\gamma)_nn!}z^n\end{gather*} is $p$-adically bounded.

 The coefficients of these two standard solutions are $\frac{(\alpha)_n(\beta)_n}{(\gamma)_nn!}$ and $\frac{(1-\alpha)_n(1-\beta )_n}{(2-\gamma)_nn!}$.

As usual $(\alpha )_n$ is the Pochhammer symbol. There is a helpful formula for $v_p((\alpha)_n)$, the number of $p$-factors in $(\alpha)_n$. Let $\alpha =a_0+a_1p+a_2p^2+\cdots $. If $a_0\neq 0$, then
\begin{gather*}
v_p((\alpha )_n)=\sum _{k\geq 1}\left[\frac{n-1+a_0+\cdots +a_{k-1}p^{k-1}}{p^k}\right] .\end{gather*}

This formula follows from the observation that for any $k\geq 1$ the number of elements in the sequence $\alpha, \alpha +1,\dots ,\alpha +n-1$ which are divisible by $p^k$ is $\big[\frac{n-1+a_0+\cdots +a_{k-1}p^{k-1}}{p^k}\big]$.

For the case $a_0=0$ one can write $v_p((\alpha)_n)=v_p(\alpha)+v_p((1+\alpha)_{n-1})$. For convenience we suppose that $a_0\neq 0$. We want a $k_0$ such that the inequality
\begin{gather*} \left[\frac{n-1+\alpha _k}{p^k}\right] + \left[\frac{n-1+\beta _k}{p^k}\right] -
\left[\frac{n-1+\gamma _k}{p^k}\right]- \left[\frac{n-1+1}{p^k}\right]\geq 0 \end{gather*}
 holds for $k\geq k_0$ and all $n$. Replace $\frac{n-1}{p^k}$ by $x$, then we want
\begin{gather*} \left[x+\frac{\alpha _k}{p^k}\right] + \left[x+\frac{\beta _k}{p^k}\right] -
\left[x+\frac{\gamma _k}{p^k}\right]- \left[x+\frac{1}{p^k}\right]\geq 0
\end{gather*}
to hold for all (real) $x$. Thus if $\max (\alpha _k,\beta _k)\geq \gamma_k$ for $k\gg 0$ the boundedness of the first set of coefficients holds. The same inequalities for $k\gg 0$ imply the boundedness of the second set of coefficients.

We now sketch a proof of the statement that $H$ and the induced stratification have the same local exponents. We consider the local exponents at $z=0$ and suppose (for convenience) that all coefficients of $F_1$ and $F_2$ are $p$-adically bounded by~1. Write
\begin{gather*} G_1=F_1 \qquad \text{and}\qquad G_2=(1-z)^{\gamma -\alpha -\beta}
 \sum _n\frac{(1-\alpha)_n(1-\beta )_n}{(2-\gamma)_nn!}z^n.\end{gather*}

The matrix differential equation $\frac{{\rm d}}{{\rm d}z} y=Ay$ for $H$ and the derived equations $\frac{1}{n!}\big(\frac{{\rm d}}{{\rm d}z}\big)^ny=A_ny$ are defined by $\frac{1}{n!}\big(\frac{{\rm d}}{{\rm d}z}\big)^n {F_1\ F_2 \choose F_1'\ F_2'}=A_n{F_1\ F_2 \choose F_1' \ F_2'}$. The $A_n$ have their coefficients in $\mathbb{Q}_p(z)$ and by assumption reduce modulo $p$ to elements of $\mathbb{F}_p(z)$. The expressions $G_1$ and $G_2$ can also be reduced modulo $p$ to matrices with entries in~$\mathbb{F}_p[[z]]$.

For any $p$-adic number $\tau \in \mathbb{Z}_p$ we see $z^\tau$ as a symbolic, non-trivial solution over the field $\mathbb{F}_p((z))$ of the iterative equation $\big\{\partial_z^{(n)}y={\tau \choose n }z^{-n}y\, | \, n\geq 0\big\}$.

The matrix ${F_1\ F_2 \choose F_1'\ F_2'}$ can also be reduced modulo $p$, its entries are in $\mathbb{F}_p[[z]][z^{1-\gamma}]$. This reduction is a fundamental matrix for the stratification. It follows at once that the local exponents at $z=0$ of the stratified equation are $0$, $1-\gamma$. \end{proof}

The computation of the $p$-adic valuations of the coefficients of the standard hypergeometric functions is due to F.~Beukers (oral communication). He claims (without proof) that the above inequalities describe completely the set of the standard classical hypergeometric equations ${\rm HG}(\alpha, \beta ,\gamma)$ over~$\mathbb{Q}_p$ which produce a bounded system (of divided equations) $\frac{1}{n!}\big(\frac{{\rm d}}{{\rm d}z}\big)^ny=A_ny$ over $\mathbb{Q}_p$. We note that the reduced matrices have their entries in $\mathbb{F}_p(z)$ and that the stratification is therefore defined over the base field~$\overline{\mathbb{F}}_p$.

Let $T\subset \mathbb{Z}_p^3$ consists of the tuples $(\alpha ,\beta, \gamma )$ such that (using the notation of Theorem~\ref{3.6}) $\max(\alpha_k,\beta_k)\geq \gamma_k$ holds for $k\gg 0$. The set~$T$ is of arithmetic nature and has a positive volume. {\it There are many natural questions}:
\begin{enumerate}\itemsep=0pt
\item[(a)] For the generic situation $(\alpha, \beta ,\gamma )\in T$ and $\alpha$, $\beta$, $\gamma$ algebraically independent over $\mathbb{Q}$, there is no (formal) difference between the $p$-adic hypergeometric equation and the complex one. Therefore the $p$-adic equation has differential Galois group ${\rm GL}_2$ over some $p$-adic field. According to \cite[Section~8]{Ma-vdP1}, the corresponding standard stratified hypergeometric equation has again differential Galois group ${\rm GL}_2$, now over the field $\overline{\mathbb{F}}_p$. What are the differential Galois groups for stratifications coming from non-generic $(\alpha,\beta,\gamma)\in T$?
\item[(b)] Is the standard hypergeometric stratification derived from $(\alpha, \beta ,\gamma )\in T$ the only one with local exponents $ 0$, $1-\gamma || 0$, $\gamma -\alpha -\beta || \alpha$, $\beta $ at $0$, $1$, $\infty$?
\item[(c)] Are there for $(\alpha, \beta, \gamma )\not \in T$ standard hypergeometric stratifications not obtained by reducing $p$-adic hypergeometric equations?
\end{enumerate}

\subsection{Inverse problem for regular singular stratifications}\label{section3.3}
A linear algebraic group $G$ over $C$ will be called {\it tame for $(X,S)$} if $X$ is a curve over $C$ (smooth, projective, irreducible), $S\subset X$ finite and there is a regular singular stratification on $X\setminus S$ with differential Galois group $G$. Equivalently, $G$ is an image of the regular singular, stratified fundamental group $\pi^{\rm str,rs}(X\setminus S,x_0)$.

\begin{Theorem}\label{3.7} Suppose that $G$ is tame for $(X,S)$. Then $G/G^o$ is an image of the tame fundamental group of $X\setminus S$ and $G^o$ is generated by its maximal tori. In particular $\mathbb{G}_a$ is not tame for $(X,S)$.
\end{Theorem}
\begin{proof} The first statement is obvious. Consider a regular singular stratification $M$ for $(X,S)$ with differential Galois group $G$. Let $f\colon (Y,T)\rightarrow (X,S)$ denote the tame covering defined by~$G/G^o$. The stratification~$f^*M$ on~$(Y,T)$ is again regular singular (see the beginning of Section~\ref{section3} and Lemma~\ref{3.1})
 and has differential Galois group~$G^o$. Let $H\subset G^o$ denote the subgroup generated by the maximal tori of~$G^o$. The group~$H$ is, according to~\cite[Proposition~13.3.10]{Sp}, a~Zariski closed subgroup of $G^o$ and $G^o/H$ is an unipotent group. We have to show that $H=\{1\}$.

{\it Suppose otherwise}, then $H$ has the additive group $\mathbb{G}_a$ as quotient. Thus $\mathbb{G}_a$ is the differential Galois group of a regular singular stratification $N$ on $(Y, T)$. The local differential Galois group for $N$ at $t\in T$ is a subgroup of a torus $\mathbb{G}_m^n$ where $n$ is the rank of $N$. It is also a subgroup of the global differential Galois group $\mathbb{G}_a$ of $N$. Therefore the local differential Galois groups are trivial and the stratification that produces $\mathbb{G}_a$ is regular on $Y$ itself. However, this {\it contradicts} the important result of \cite[Corollary~16]{Sa} which states that the maximal unipotent quotient $\pi _Z^{\rm uni}$ of the stratified fundamental group of any complete variety $Z$ is pro-\'etale. In particular, $\mathbb{G}_a$ cannot be a quotient of $\pi^{\rm str}(Y,y_0)$. \end{proof}

\begin{Examples} For $\big(\mathbb{P}^1,S\big)$ with $\#S=1$ there are no tame groups (except $\{1\}$).

For $\#S=2$ the tame groups are $\operatorname{Diag}(\mathbb{X})$ with $\mathbb{X}$ finitely generated $\mathbb{Z}$-module without $p$-torsion. For $\#S=3$, the example ${\rm GL}_2(C)$ is produced by Theorem~\ref{3.6}.

In the special case $S=\{0,1,\lambda ,\infty\}$, $C$ complete with respect to a non trivial valuation and $0<|\lambda |<1$, there are many examples of tame groups. For instance any linear algebraic group $G$, topologically generated by two elements of finite order prime to $p$, is tame for $\big(\mathbb{P}^1,\{0,1,\lambda ,\infty\}\big)$. Indeed, this follows by combining Proposition~\ref{5.3}(2)(i) with Theorem~\ref{5.1}.
\end{Examples}

 \begin{Observation}\label{3.8} Let $G$ be a linear algebraic group over $C$.
 \begin{enumerate}\itemsep=0pt
\item[(a)] Suppose $C=\overline{\mathbb{F}}_p$. If $G$ is topologically finitely generated then~$G$ is finite.
\item[(b)] Suppose $C\neq \overline{\mathbb{F}}_p$. Then $G$ is topologically finitely generated if and only if $G^o$ is generated by its maximal tori.
\end{enumerate}
\end{Observation}

 \section{Stratifications on Mumford curves}\label{section4}
 In this section the field $C$ is supposed to be complete with respect to a non trivial valuation and, as before, $C=\overline{C}$ has characteristic $p>0$.

 For a rigid space $X$ over $C$ one defines the rigid sheaf of differential operators $\mathcal{D}_X^{\rm rigid}$ by copying the definitions in the algebraic case, see \cite[Section~16.8.1]{EGA4} and~\cite{Gie}. The only new issue is that the sheaf $\mathcal{D}_X^{\rm rigid}$ is glued from the affinoid case to the general case by the rigid topology on~$X$, instead of the Zariski topology in the algebraic context.

 For the standard multidisk $T_d:=\operatorname{Spm}(C\langle z_1,\dots ,z_d\rangle)$ one verifies as in \cite[Th\'eor\`eme~16.11.2]{EGA4} and~\cite{Gie}, that the algebra of differential operators is a free left $O(T_d)=C\langle z_1,\dots ,z_d\rangle $-module on the basis $\big\{\partial^{\underline{m}}\big\}$. Here $\underline{m}=(m_1,\dots ,m_d)$ and the $C$-linear, continuous action of $\partial^{\underline{m}}$ on $C\langle z_1,\dots , z_d\rangle$ is given by the formula
 \begin{gather*} \partial^{\underline{m}}\big(z_1^{k_1}\cdots z_d^{k_d}\big)= {k_1\choose m_1 }\cdots {k_d\choose m_d} z_1^{k_1-m_1}\cdots z_d^{k_d-m_d}.\end{gather*}
 A rigid stratified bundle on $X$ is a locally free $O_X$-module of finite rank on $X$ provided with a compatible left action by $\mathcal{D}^{\rm rigid}_X$. Let $\operatorname{Strat}^{\rm rigid}(X)$ denote the Tannaka category of the rigid stratified bundles on~$X$.

It is not difficult to verify the following properties:
\begin{enumerate}\itemsep=0pt
\item[(a)] $\mathcal{D}^{\rm rigid}_X$ is a quasi coherent $O_X$-module, locally free of countable type.
\item[(b)] Let $X$ be a smooth algebraic variety over $C$ and let $X^{\rm an}$ denote its analytification. Then $\mathcal{D}_{X^{\rm an}}^{\rm rigid}$ is the analytification of the algebraic sheaf $\mathcal{D}_X$.
\item[(c)] If $X$ is a smooth, projective variety, then ``analytification'' is an equivalence $\operatorname{Strat}(X)\rightarrow \operatorname{Strat}^{\rm rigid}\big(X^{\rm an}\big)$ of Tannakian categories. The key point in the proof is the ``GAGA'' theorem in the rigid context.
\item[(d)] In general, for a smooth variety $X$ over $C$, the functor $\operatorname{Strat}(X)\rightarrow \operatorname{Strat}^{\rm rigid}\big(X^{\rm an}\big)$ need not be full and need not be surjective on objects. For example, for $X=\mathbb{A}^1$ a non trivial stratified module can become trivial in the rigid sense. Moreover, not every rigid stratification comes from an algebraic one. We note that this is similar to the complex case.
\end{enumerate}

\begin{Example}[a non trivial stratification on the unit disk $\operatorname{Spm}(C\langle z\rangle)$]\label{4.1} This deviation from the complex situation has a consequence for stratifications on a Mumford curve, namely part~(2) of Theorem~\ref{4.2}.

The affinoid algebra of functions on the unit disk is $C\langle z\rangle$ and consists of the power series $\sum a_nz^n$ with $\lim |a_n|=0$. Consider $e=\prod\limits_{m\geq 0}\big(1+ c_mz^{p^m}\big)\in C[[z]]\supset C\langle z\rangle$ and such that $1/p \leq |c_m|<1$ for all $m$. We will verify that $\partial^{(n)}e=a_ne$ with all $a_n\in C\langle z\rangle $. The free module $E:=C\langle z\rangle b$ is made into a stratification by $\partial^{(n)}b=a_nb$ for all~$n$. A solution is $f\cdot b$ with $f=\frac{1}{e}$. One can choose the $\{c_m\}$ such that $e$ is transcendental over $C\langle z\rangle$. Then the stratification $E$ is not trivial and its differential Galois group is $\mathbb{G}_m$.

 The verification: For fixed $n$ there is an integer $k>0$ such that
\begin{gather*} \partial^{(n)}e= \left(\partial^{(n) }\prod_{m=0}^{k-1} \big(1+c_mz^{p^m}\big)\right)\cdot \prod_{m\geq k}\big(1+c_mz^{p^m}\big).\end{gather*}
 Thus \begin{gather*} a_n=\frac{\partial^{(n)}e}{e}=\frac{\partial^{(n)}\prod\limits_{m=0}^{k-1} \big(1+c_mz^{p^m}\big)}{\prod\limits_{m=0}^{k-1} \big(1+c_mz^{p^m}\big)}\end{gather*} and this belongs to $C\langle z\rangle$. \end{Example}

A {\it Schottky group $\Delta$} is a finitely generated discontinuous subgroup of ${\rm PGL}_2(C)$ which has no element $\neq 1$ of finite order. The subset $\Omega$ of $\mathbb{P}^1(C)$ consisting of the ordinary points for $\Delta$ is an open rigid subspace. The quotient $X:=\Omega/\Delta$ is (the analytification of) a smooth projective algebraic curve of a certain genus $g\geq 1$. The case $g=1$ corresponds to $\Delta =\big\langle {q\ 0\choose 0\ 1 }\big\rangle$ with $0<|q|<1$, $\Omega =\mathbb{P}^1(C)\setminus \{0,\infty\}$ and~$X$ is the Tate curve.

For $g\geq 2$, the group $\Delta$ is a free, non commutative, group on $g$ generators and $X$ is called a~Mumford curve. See \cite{Ge-vdP} for more details.

Let $\operatorname{Repr}_\Delta$ denote the Tannaka category of the representations of $\Delta$ on finite-dimensional $C$-vector spaces. As before, $\operatorname{Strat}(X)$ denotes the Tannaka category of the stratified bundles on~$X$. The importance for the inverse problem is the following theorem, also present in~\cite{Sa}.

\begin{Theorem}\label{4.2}\quad
\begin{enumerate}\itemsep=0pt
\item[$(1)$] The rigid uniformization $u\colon \Omega \rightarrow X$ induces a fully faithful functor $\mathcal{F}\colon \operatorname{Repr}_\Delta \rightarrow \operatorname{Strat}(X)$ of Tannakian categories.
\item[$(2)$] An object $M$ of $\operatorname{Strat}(X)$ lies in the essential image of $\mathcal{F}$ if and only if the analytifica\-tion~$M^{\rm an}$ of~$M$ is locally trivial for the rigid topology on the analytification $X^{\rm an}$ of~$X$.
\item[$(3)$]Let $\rho\colon \Delta \rightarrow {\rm GL}(V)$ be a representation. The differential Galois group of the stratified bundle $\mathcal{F}(\rho)$ is the Zariski closure of $\rho(\Delta)$ in ${\rm GL}(V)$.
\item[$(3)$]A linear algebraic group $G$ over $C$ is a differential Galois group for $X$ if~$G$ is topologically generated $($for the Zariski topology$)$ by $\leq g$ elements.
\end{enumerate}
\end{Theorem}

\begin{Observations}\label{4.3} \quad
\begin{enumerate}\itemsep=0pt
\item[(1)] Theorem \ref{4.2}(1), is equivalent to the existence of a surjective morphism of affine group schemes $\pi^{\rm str}(X,x_0)\rightarrow (\Delta)^{\rm hull}$. The latter affine group scheme is the algebraic hull of $\Delta$ with respect to the field $C$. This algebraic hull can be described as the projective limit of $\overline{\rho(\Delta)}$, where~$\rho$ runs in the set of representation $\rho\colon \Delta \rightarrow {\rm GL}(V)$ on finite-dimensional vector spaces over~$C$ and $\overline{\rho(\Delta)}$ denotes the Zariski closure of $\rho(\Delta)$.
\item[(2)]In \cite{Sa} a nice expression for the stratified fundamental group of a Tate curve $E$ is derived from Theorem~\ref{4.2}.
\item[(3)]For a curve $X$ over $\mathbb{C}$ of genus $g\geq 1$, the Riemann--Hilbert correspondence produces an equivalence between the representations of the topological fundamental group of $X$ (which has $2g$ generators and one relation) and the category of the (regular) connections on~$X$. For a Mumford curve $X$ over $C$, the rigid fundamental group $\Delta$ of $X$ (which has $g$ generators and no relation) only captures ``half of the expected fundamental group''.
\item[(4)]Let a stratified bundle on $X$ be given. The pull back under $\Omega \rightarrow X$ produces a trivial (i.e., free) rigid bundle with a (rigid) stratification. In general this rigid stratification is not trivial (see Example~\ref{4.1}). This explains part~(2) of Theorem~\ref{4.2} and observation~(3).
\end{enumerate}
\end{Observations}

\begin{proof}We sketch the proof of Theorem~\ref{4.2}. A similar reasoning will be used in the proof of Theorem~\ref{5.1}. Let $\rho\colon \Delta \rightarrow {\rm GL}(V)$ be a homomorphism and $V$ is a finite-dimensional vector space over~$C$. On $\Omega$ we consider the trivial vector bundle $\Omega \times V\rightarrow \Omega$. This is provided with the trivial stratification, i.e., the constant sections $\omega \mapsto (\omega ,v)$ of $\Omega \leftarrow \Omega \times V$ form the solution space of the stratification. Now we define an action of $\Delta$ on $\Omega \times V$ by $\delta (\omega, v)=(\delta (\omega), \rho(\delta )v)$ for all $\delta \in \Delta$. This action commutes with the stratification on $\Omega \times V$. Dividing by $\Delta$ yields a rigid vector bundle $(\Omega \times V)/\Delta \rightarrow \Omega /\Delta =X^{\rm an}$ with a stratification. By GAGA, this rigid stratified bundle is the analytification of a unique algebraic stratification on~$X$. We write $\mathcal{F}(\rho)$ for this stratification on~$X$ and note that $\rho \mapsto \mathcal{F}(\rho)$ is an analogue of the complex Riemann--Hilbert correspondence. It is easy to verify that $\rho \mapsto \mathcal{F}(\rho)$ defines a functor which respects all constructions of linear algebra, and is fully faithful. This proves part~(1).

If $M\in \operatorname{Strat}(X)$ is in the essential image of $\mathcal{F}$, then, by construction, $u^*M$ is a trivial stratification. On the other hand, if the stratification $u^*M$ on $\Omega$ is locally trivial for the rigid topology, then~$u^*M$ is globally trivial since $\Omega$ does not admit proper rigid coverings. This proves part~(2).

Part (3) follows in fact from Observations \ref{4.3}(1) and part (4) is a consequence of part~(3).

A proof of part (3), using Picard--Vessiot theory, is the following. Let $M=\mathcal{F}(\rho)$, then~$u^*M$ is a trivial stratification with, by construction, has a~solution space $V$ and, after fixing a basis of~$M$, has a~fundamental matrix~$F$ with entries in the field $C(\Omega)$ of the meromorphic functions on~$\Omega$. The subfield~$PV$ of~$C(\Omega)$, generated over~$C(X)$ by the entries of~$F$, $\frac{1}{\det F}$ and all their higher derivatives is a Picard--Vessiot field for~$M$, because its field of constants is~$C$. The group $\Delta$ acts on $C(\Omega)$ and on the subfield~$PV$. Its action on the fundamental matrix $F$ is $\delta (F)=F\cdot \rho(\delta)$ for any $\delta \in \Delta$. From $C(\Omega )^\Delta=C(X)$ it follows that $PV^\Delta =C(X)$. This implies that $\rho(\Delta)$ is Zariski dense in the differential Galois group.
\end{proof}

\begin{Examples}\label{4.4} For a Mumford curve of genus~2, the groups ${\rm SL}_2(C)$ and $\big\{{a\ b \choose 0\ 1 }\,|\, a\in C^*,b\in C\big\}$ are differential Galois groups. Choose an $\alpha \in C^*$ which is not a root of unity. The first group is topologically generated by ${\alpha\ 0\choose 0\ \frac{1}{\alpha}}$ and ${1\ 1\choose 0\ 1 } {\alpha\ 0\choose 0\ \frac{1}{\alpha} }{1\ -1\choose 0\ 1 }$. The second group is topologically gene\-ra\-ted by ${\alpha\ 0\choose 0\ 1 }$ and ${\alpha \ 1\choose 0\ 1}$. Apply now Theorem~\ref{4.2}(3). We note in connection with Theorem~\ref{3.6} that {\it the second group is generated by its maximal tori but is not a reductive group}. For a generic curve of genus two one can state the following:

 {\it Suppose $p>2$. $C_0$ denotes the algebraic closure of the field $\mathbb{F}_p(A_1,A_2,A_3)$ where $A_1$, $A_2$, $A_3$ are algebraically independent over~$\mathbb{F}_p$. Let $X_0$ over $C_0$ be ``generic genus two curve'' $y^2=x(x-1)(x-A_1)(x-A_2)(x-A_3)$. Let $G$ be a linear algebraic group,
 topologically generated by two elements.

 There is a stratification on $C\times _{C_0}X_0$ for some $($algebraically closed$)$ field extension $C$ of $C_0$, with differential Galois group $G$.}

 Indeed, since the $A_1$, $A_2$, $A_3$ are algebraically independent there exists a valuation $| \ |$ on~$C_0$ such that $0<|A_1|<1$, $0<|A_2-1|<1$, $|A_3|>1$. Let $C$ be the completion of $C_0$ and write $X=C\times _{C_0}X_0$. From the position of the six ramification points $0$, $1$, $A_1$, $A_2$, $A_3$, $\infty$ in $\mathbb{P}^1(C)$ one can deduce that $X$ is a Mumford curve (see \cite[p.~168]{Ge-vdP}). Then apply, as above, Theorem~\ref{4.2}(3). In general (except for finite groups~$G$0 this stratification over $C$ descends to a field which is countably generated over~$C_0$, but does not descend to~$C_0$ itself.

 For any $g\geq 1$ one can consider the ``generic curve'' $X_0/C_0$ of genus $g$. For a suitable valuation on $C_0$ and its completion $C$, the curve $X=C\times_{C_0}X_0$ is a Mumford curve and any linear algebraic group $G$, topologically generated by $\leq g$ elements, is the differential Galois group for some stratification on~$X$. Compare also Proposition~\ref{1.1}.
 \end{Examples}

\section{Mumford groups and stratified bundles}\label{section5}

As in Section~\ref{section4}, $C$ is an algebraically closed field of characteristic $p>0$, complete with respect to a non trivial valuation. We recall some definitions and facts from \cite[Section~9]{Vo-vdP}. A subgroup $\Gamma \subset {\rm PGL}_2(C)$ is called a {\it Mumford group} if $\Gamma$ is finitely generated, discontinuous and such that $\Omega /\Gamma\cong \mathbb{P}^1$, where $\Omega \subset \mathbb{P}^1(C)$ is the rigid open subset of the ordinary points for $\Gamma$. We exclude the cases: $\Gamma$ is finite and $\Gamma$ has a subgroup of finite index, isomorphic to $\mathbb{Z}$.

According to \cite[Theorem I.3.1]{Ge-vdP}, the group $\Gamma$ contains a normal subgroup $\Delta$ of finite index such that $\Delta$ has no elements $\neq 1$ of finite order. Such a~$\Delta$ is a~Schottky group, free non abelian on $g \geq 2$ generators.

The image $\overline{\omega}\in \mathbb{P}^1$ of a point $\omega \in \Omega$ with non trivial stabilizer $\Gamma_\omega\subset \Gamma$ is called a {\it branch point} for $\Gamma$. If $\overline{\omega _1}=\overline{\omega_2}$, then $\Gamma_{\omega_1}$ and $\Gamma_{\omega_2}$ are conjugated subgroups of $\Gamma$. The branch point $\overline{\omega}$ is called {\it tame} if the order of $\Gamma_\omega$ is prime to $p$. In the tame case $\Gamma_\omega$ is cyclic.

Any representation $\rho\colon \Gamma \rightarrow {\rm GL}(V)$ produces, by the method explained in Section~\ref{section4}, a~stra\-ti\-fication on $\mathbb{P}^1$ with singularities at the branch points. This stratification is regular singular at a branch point if and only if that point is tame. This is in accordance with~\cite{Ki}. We call a~Mumford group {\it tame} if every branch point is tame.

{\it We fix a tame Mumford group $\Gamma$} with branch points $a_1,\dots ,a_r$ and ramification indices $e_1,\dots ,e_r$ and thus $(p, e_1\cdots e_r)=1$. Let $\operatorname{Strat}\big(\mathbb{P}^1,\sum e_i[a_i]\big)$ denote the Tannakian category of the regular singular stratifications on $\mathbb{P}^1\setminus \{a_1,\dots ,a_r\}$ such that, for each $i$, the local exponents at $a_i$ (taken modulo $\mathbb{Z}$) belong to $\frac{1}{e_i}\mathbb{Z}/\mathbb{Z}$.

{\it We fix a normal Schottky subgroup} $\Delta$ of finite index in $\Gamma$ such that the Galois covering $Z:=\Omega/ \Delta \rightarrow \mathbb{P}^1$ with group~$\Gamma/\Delta$ is tamely ramified above $a_1,\dots ,a_r$ with ramification indices $e_1,\dots ,e_r$.

 \begin{Theorem}\label{5.1}\quad
 \begin{enumerate}\itemsep=0pt
\item[$(1)$] The morphism $u\colon \Omega \rightarrow \Omega/\Gamma=\mathbb{P}^1$ induces a fully faithful functor $\mathcal{F}\colon \operatorname{Repr}_\Gamma \rightarrow \operatorname{Strat}\big(\mathbb{P}^1,\allowbreak \sum e_i[a_i]\big)$.
 \item[$(2)$] An object $M$ in $\operatorname{Strat}\big(\mathbb{P}^1,\sum e_i[a_i]\big)$ is in the essential image of $\mathcal{F}$ if and only if the analytification of its pull back on $Z$ is locally (for the rigid topology) a trivial stratification.
 \item[$(3)$] Let $\rho\colon \Gamma \rightarrow {\rm GL}(V)$ be a representation. The differential Galois group of the stratifica\-tion~$\mathcal{F}(\rho)$ is the Zariski closure of the image of $\rho$ in ${\rm GL}(V)$.
 \end{enumerate}
 \end{Theorem}

 Consider the situation $C((z))\subset C((t))$ with $z=t^e$ and $(p,e)=1$. Now $G=\langle \sigma \rangle$, with $\sigma (t)=\zeta t$ and $\zeta \in C$ a primitive $e$th root of unity, is the automorphism group of $C((t))/C((z))$. Let $\big\{\partial _z^{(n)}\big\}$ denote the standard higher derivation on~$C((z))$ and $\big\{\partial _t^{(n)}\big\}$ the standard higher derivation on $C((t))$. For the proof of Theorem~\ref{5.1} one need the following result.

 \begin{Lemma}\label{5.2}\samepage\quad
\begin{enumerate}\itemsep=0pt
\item[$(a)$] Let $M$ be a trivial stratified module over $C((t))$ provided with a compatible $G$-action. Then the stratified module $M^G$ over $C((z))$ is regular singular and its exponents are in $\frac{1}{e}\mathbb{Z}/\mathbb{Z}\subset \mathbb{Z}_p/\mathbb{Z}$.
\item[$(b)$]Let $N$ be a regular singular stratified module over $C((z))$ with exponents in $\frac{1}{e}\mathbb{Z}/\mathbb{Z}$. Then $M=C((t))\otimes N$ is a trivial stratified module over $C((t))$ provided with a compatible $G$-action.
\end{enumerate}
 \end{Lemma}
\begin{proof} (a) $M$ has the form $C((t))\otimes _CV$ where $V$ is a finite-dimensional $C$-vector space and the higher derivatives $\partial _M^{(k)}$, $k\geq 1$ are zero on $V$. In other words, $V$ is the solution space of $M$. Hence~$V$ is invariant under $G$. One writes~$V$ as a direct sum of $G$-invariant, 1-dimensional spaces and~$M$ is then direct sum of spaces of the form $\tilde{M}:=C((t))m$ with $\partial _M^{(k)}m=0$ for all $k\geq 1$ and $\sigma m=\zeta ^im$ with $1\leq i \leq e$. Then $\tilde{M}^G=C((z))t^{e-i}m$ and $\partial _z^{(k)}(t^{e-i}m)={\frac{e-i}{e}\choose k }z^{-k}t^{e-i}m$. This proves~(a).

(b) It suffices to consider the 1-dimensional case $N=C((z))n$ with $\partial_N^{(k)}n={\alpha \choose k }z^{-k}n$ and $e\cdot \alpha \in \mathbb{Z}$. Then $M:=C((t))\otimes N$ and, for suitable $i$, the basis vector $t^in$ of~$M$ satisfies $\partial_M^{(k)} t^in=0$ for all $k\geq 1$. \end{proof}

A proof of Theorem~\ref{5.1} is obtained by combining Lemma \ref{5.2} with the proof of Theorem~\ref{4.2}.

For the existence of tame Mumford groups one can use the classification of all Mumford groups with 2 or 3 branch points and the fact that a tame Mumford group has no elements of order~$p$ (see~\cite{Vo-vdP}). This leads to

 \begin{Proposition} \label{5.3} \quad
\begin{enumerate}\itemsep=0pt
\item[$(1)$] There is no tame Mumford group with $3$ branch points.
\item[$(2)$] The tame Mumford groups with $4$ branch points are the realizations of:
\begin{enumerate}\itemsep=0pt
\item[$(i)$] $C_\ell \ast C_m$ with $p\,{\not |}\, \ell m$ and $\ell, m >1$ and
\item[$(ii)$] $G_1\ast _{G_3} G_2$ where $G_1,G_2\in \{D_\ell ,A_4,S_4,A_5\}$ have orders not divisible by $p$ and $G_3$ is a~branch group in both $G_1$ and $G_2$.
\end{enumerate}
\item[$(3)$] In case $(i)$, the ramification indices are $(\ell ,\ell , m,m)$.
\end{enumerate}
In case $(ii)$, the group $G_3$ is necessarily a cyclic group of order $\ell$. If $G_1$ and $G_2$ have branch points with ramification indices $(n_1,n_2,\ell)$ and $(m_1,m_2,\ell)$, then $G_1\ast_{G_3}G_2$ has ramification indices $(n_1,n_2,m_1,m_2)$.
\end{Proposition}

\begin{Comments}\label{5.4} Part (1) of Proposition~\ref{5.3} can be seen in several ways (see (1) and (2) below).
\begin{enumerate}\itemsep=0pt
\item[(1)] Let $\operatorname{br}(\Gamma)$ denote the number of branch points for $\Gamma$. If $\Gamma$ is irreducible and tame, then, according to \cite[Theorem~3.1 and Section~9.1]{Vo-vdP}, $\Gamma$ has no elements of order $p$ and $\operatorname{br}(\Gamma)=3\cdot \# {\rm vertices}-2\#{\rm edges}=2+\# {\rm vertices}\geq 4$. For a reducible $\Gamma=\Gamma_1\ast \Gamma_2$ one has $\operatorname{br}(\Gamma)=\operatorname{br}(\Gamma_1)+\operatorname{br}(\Gamma_2)\geq 4$. These formulas produce in fact all cases with 4 tame branch points.
\item[(2)] Let the Mumford group $\Gamma$ be tamely ramified over $S\subset \mathbb{P}_C^1$. A Schottky group $\Delta \subset \Gamma$ (normal and of finite index) produces a Galois covering $X\rightarrow \mathbb{P}^1_C$ which is tamely ramified over $S$. If $S$ is defined over the algebraic closure of $\mathbb{F}_p$ in~$C$, then it is well known~\cite[Expos\'e~IX]{SGA1} that $X$ is also defined over the algebraic closure of~$\mathbb{F}_p$. It follows that $X$ has good reduction and cannot be a Mumford curve. This implies part~(1) of Proposition~\ref{5.3} for $S=\{0,1,\infty\}$.
The same argument shows that the 4 branch points in $\mathbb{P}^1(C)$ for the groups in part~(2) of Proposition~\ref{5.3} cannot reduce to~4 distinct points on the projective line over the residue field of~$C$. Therefore the 4 branch points can be put into the position $\{0,1,\infty, \lambda \}$ with $0<|\lambda |<1$.

\item[(3)] Let $\Gamma\subset {\rm PGL}_2(C)$ denote a tame Mumford group and $S\subset \mathbb{P}^1(C)$ its set of branch points. A~linear algebraic group~$G$ over~$C$ will be called {\it $\Gamma$-realizable} if there is a homomorphism $h\colon \Gamma \rightarrow G$ with dense image. By Theorem~\ref{5.1} this implies that there is a regular singular stratification on $\mathbb{P}^1$ with singularities in $S$ and with differential Galois group~$G$.
\end{enumerate}
\end{Comments}

\begin{Corollary}\label{5.5} Let $\Gamma\subset {\rm PGL}_2(C)$ denote a realization of the tame amalgam $C_\ell\ast C_m$ with $p\, {\not |}\, \ell m$ and $(\ell,m)\neq (2,2)$. Let $S=\{0,1,\lambda ,\infty\}$ $(0<|\lambda |<1)$ be the set of its branch points. Let $G$ be any linear algebraic group over $C$ which is topologically finitely generated. Then:
\begin{enumerate}\itemsep=0pt
\item[$(1)$] There is a Galois covering $f\colon X\rightarrow \mathbb{P}^1$, tamely ramified over~$S$, such that $X$ is a Mumford curve and $X$ has a regular stratification with differential Galois group~$G$.
\item[$(2)$] There is a $\Gamma$-realizable group $G_1$ and there are $($reduced$)$ algebraic groups $G_3\subset G_2\subset G_1$ such that~$G_3$ is normal in $G_2$ and $G_2$ is normal and of finite index in~$G_1$ and $G\cong G_2/G_3$.
\end{enumerate}
 \end{Corollary}

 \begin{proof} Since $(\ell ,m)\neq (2,2)$, the group $\Gamma$ has a normal Schottky subgroup $E$, which is free on $g\geq 2$ generators. Subgroups $\Delta$ of finite index in $E$ are free on $1+(g-1)[E:\Delta]$ generators. Furthermore~$E$ contains subgroups $\Delta$, normal in $\Gamma$, of arbitrary large finite index. Hence, for every $d\geq 1$, there is a normal Schottky subgroup $\Delta$ of finite index which is free on at least~$d$ generators. For~$d$ large enough there is a homomorphism $\Delta \rightarrow G$ with Zariski dense image. Now apply Theorem~\ref{4.2} to an embedding $G\subset {\rm GL}(W)$. This proves~(1).

 The representation $\phi\colon \Delta \rightarrow G\subset {\rm PGL}(W)$ yields an induced representation $\psi\colon \Gamma \rightarrow {\rm PGL}(V)$ and, with the usual notation, one has $V=\oplus _{\gamma \in \Gamma /\Delta}W_\gamma$. Let $G_1$ be the Zariski closure of $\psi (\Gamma)$ and $G_2$ the Zariski closure of $\psi(\Delta)$. Each subspace $W_\gamma$ is invariant under~$G_2$ and the image of~$G_2$ in ${\rm GL}(W_1)$, for its action on the subspace~$W_1$, coincides with $G$. This proves~(2).

 We note that the pushforward under $f$ of the stratification on $X$ in part (1) is the stratification on $\mathbb{P}^1$ with group $G_1$.
 \end{proof}

\begin{Observations}\label{5.6}\quad
\begin{enumerate}\itemsep=0pt
\item[(1)] One has $C\neq \overline{\mathbb{F}_p}$ since the valuation on $C$ is not trivial. Then, by Observation \ref{3.8}, a~linear algebraic group $G$ is topologically finitely generated if and only if $G^o$ is generated by its maximal tori.
\item[(2)] An analogue of Corollary~\ref{5.5} holds for any tame Mumford group.
\item[(3)] {\it Let $G$ be any finite group}. In the statements of Corollary~\ref{5.5} one can replace the ``big'' field $C$ by the subfield $F:=\overline{\mathbb{F}_p(\lambda)}$. Indeed, any covering $Z\rightarrow \mathbb{P}_C^1$, tamely ramified above $S=\{0,1,\infty ,\lambda\}$ is in fact defined over a~finitely generated $F$-algebra~$R$. After dividing $R$ by a suitable maximal ideal, one obtains a Galois covering with the same group over the projective line over~$F$, tamely ramified over 4 points $0$, $1$, $\infty$, $\lambda$. By Corollary~\ref{5.5}, {\it any finite group $G$ is a subquotient of the tame fundamental group of $\mathbb{P}^1\setminus \{0,1,\lambda ,\infty \}$ if $\lambda \not \in \overline{\mathbb{F}_p}$.}
More precisely: {\it Let $F$ be an algebraically closed field of transcendence degree $\geq 1$ over $\mathbb{F}_p$ and suppose that $S\subset \mathbb{P}^1_F$ is a set of $4$ elements in general position, i.e., a cross-ratio is not in~$\overline{\mathbb{F}_p}$. For any finite group $G$ there are groups $G_3\lhd G_2\lhd G_1$, $G_1$ is a quotient of the tame fundamental group of $\mathbb{P}^1_F\setminus S$ and $G\cong G_2/G_3$.}

\item[(4)] A finite quotient $G$ of $C_\ell \ast C_m$ produces a tame cover of $\mathbb{P}^1$ ramified over $0$, $1$, $\lambda$, $\infty$, defined over the field $\overline{\mathbb{F}_p(\lambda)}$. We give an {\it example} with $p|\ \#G$. Let $q$ be a power of the prime $p$ and suppose that $\ell$ and~$m$ divide $q-1$. Let $a, b\in \mathbb{F}_q^*$ have orders~$\ell$ and~$m$. Let $G\subset {\rm GL}_2(\mathbb{F}_q)$ denote the group generated by the matrices ${a\ 0\choose 0\ 1},{b\ 1\choose 0\ 1 }$. The group $G$ is an image of $C_\ell \ast C_m$ and the commutator of the generators is ${1\ a-1\choose 0\ \ \ 1 }$ and has order $p$.
\end{enumerate}
\end{Observations}

\begin{proof}[The proof of Proposition \ref{1.1}]
(1) It suffices to consider the case $A_0$ is the trivial stratification. With some obvious notations,
 the statement in that case is the bijectivity of the map
\begin{gather*}
C\otimes_{C_0}\big\{b\in B_0\,|\, \partial ^{(n)}b=0 \ \text{for all} \ n\geq 1\big\}\rightarrow \big\{ b\in B \,|\, \partial^{(n)}b=0
\ \text{for all} \ n\geq 1\big\}.
\end{gather*}
Choose a basis $\{\xi_i\}_{i\in I}$ of the $C_0$-linear space $C$. Write $b\in B$ as $\sum _ib_i\xi_i$ with all $b_i\in B_0$. From $\partial^{(n)}b=\sum _i \partial^{(n)}(b_i)\xi_i$ the statement follows.

(2) As usual, the Picard--Vessiot field of $A_0$ is the field of fractions of $C_0(Y_0)\big[\{X_{i,j}\},\frac{1}{\det}\big]/P$, where the operators $\partial ^{(n)}$ are acting like $\big(\partial ^{(n)}X_{i,j}\big)=A_n\cdot (X_{i,j})$ with prescribed matrices $A_n$ having entries in $C_0(Y_0)$ and $P$ is a maximal differential ideal. A Picard--Vessiot field for $A$ is obtained as the field of fractions of $C(Y)\big[\{X_{i,j}\},\frac{1}{\det}\big]/Q$ where $Q$ is a maximal differential ideal containing the ideal $(P)$ generated
 by $P$. It suffices to show that $Q=(P)$.

Suppose that $Q\neq (P)$. Consider a finitely generated $C_0$-algebra $R\subset C$ and the inclusion $\psi\colon R\otimes _{C_0}C_0(Y_0)\big[\{X_{i,j}\}, \frac{1}{\det}\big]\rightarrow C(Y)\big[\{X_{i,j}\},\frac{1}{\det}\big]$. For suitable $R$ the prime ideal $\psi^{-1}Q$ gene\-ra\-tes $Q$. Now $\psi^{-1}Q$ is a differential ideal, strictly larger than the ideal generated by $P$. After dividing $R$ by a suitable maximal ideal one obtains in $C_0(Y_0)\big[\{X_{i,j}\},\frac{1}{\det}\big]$ a differential ideal which is strictly larger than $P$. This contradicts the maximality of $P$.

(3) The abelianized stratified fundamental group of a curve $X_0$ (irreducible, smooth, complete) over~$C_0$ coincides with the stratified fundamental group of its Jacobian variety $\operatorname{Jac}(X_0)$. By \cite[Theorem~21]{Sa} (see also Theorem~\ref{2.4}) the stratified fundamental group of $\operatorname{Jac}(X)$ is much larger than that of $\operatorname{Jac}(X_0)$.

Furthermore Proposition~\ref{2.5} and a combination of Proposition~\ref{2.6} and Corollary~\ref{2.7} yield examples for part~(3) of Proposition~\ref{1.1}.
\end{proof}

\subsection*{Acknowledgements} We thank Andreas Maurischat for pointing out that his thesis contains a negative answer for Question~\ref{Question2}. Many thanks are due to the referees for their work which is the basis for a major revision of the text.

\pdfbookmark[1]{References}{ref}
\LastPageEnding

\end{document}